\documentclass[10pt, twoside, notitlepage]{amsart}
\usepackage[a4paper, left=2.5cm, right=2.5cm, top=3cm, bottom=3cm]{geometry}

\usepackage[utf8]{inputenc}
\usepackage{xcolor}
\usepackage{amsmath}
\usepackage{amssymb}
\usepackage{amsthm}
\usepackage{geometry}
\usepackage{graphicx}
\usepackage{esint}
\usepackage[ocgcolorlinks,linkcolor=blue]{hyperref}
\usepackage{backref}

\usepackage{pgfplots}
\pgfplotsset{compat=1.15}
\usepackage{mathrsfs}
\usetikzlibrary{arrows}
\definecolor{orange}{rgb}{1.0, 0.5, 0.0}

\definecolor{uuuuuu}{rgb}{0.26666666666666666,0.26666666666666666,0.26666666666666666}

\theoremstyle{plain}
\newtheorem{thm}{Theorem}
\newcounter{counter}

\newtheorem{theorem}[counter]{Theorem}
\newtheorem{prop}{Proposition}[section]
\newtheorem{lem}{Lemma}

\newtheorem{rmk}[prop]{Remark}

\newtheorem{question}[prop]{Question}

\newcommand {\R} {\mathbb{R}} \newcommand {\Z} {\mathbb{Z}}
 
\newcommand {\C} {\mathbb{C}}

\DeclareMathOperator{\dis} {\operatorname{d}}

\numberwithin{equation}{section}

\allowdisplaybreaks

\title{Landis' conjecture: a survey}

\author[A. Fern\'andez-Bertolin]{Aingeru Fern\'andez-Bertolin}
\address[Aingeru Fern\'andez-Bertolin]{Universidad del Pa\'is Vasco / Euskal Herriko Unibertsitatea,
	48080 Bilbao, Spain}
\email{aingeru.fernandez@ehu.eus}
\author[L. Roncal]{Luz Roncal}
\address[Luz Roncal]{BCAM -- Basque Center for Applied Mathematics,
	48009 Bilbao, Spain,
	\newline \phantom{\quad} \&
	Universidad del Pa\'is Vasco / Euskal Herriko Unibertsitatea,
	48080 Bilbao, Spain,
	\newline \phantom{\quad}\&
	Ikerbasque, Basque Foundation for Science,
	48011 Bilbao, Spain,}
\email{lroncal@bcamath.org }
\author[D. Stan]{Diana Stan}
\address[Diana Stan]{University of Cantabria, Department of Mathematics, Statistics and Computation,
Avd. Los Castros 44, 39005 Santander, Spain}
\email{diana.stan@unican.es}



\thanks{This version has appeared in: Harmonic Analysis and Nonlinear Partial Differential Equations, 35--78, RIMS K\^oky\^uroku Bessatsu, B99, Res. Inst. Math. Sci. (RIMS), Kyoto, 2026}

\begin{document}

\begin{abstract}
     We survey Kondrat'ev--Landis' conjecture, providing an up-to-date account of the main advances and describing the techniques developed. We complement the overview with references and formulations of the problem in further closely connected contexts. 
\end{abstract}

\maketitle

\tableofcontents

\section{Introduction}
\label{sec:Intro}

In 1988, Vladimir Aleksandrovich Kondrat'ev and Evgeni\v{\i} Mikha\v{\i}lovich Landis provided an up-to-date survey \cite{KL88} on the most important results of the qualitative theory of second-order elliptic and parabolic equations obtained by various authors over the previous years. In \cite[\S3.5]{KL88} (see also \cite[\S3.5]{KL91}), the following question was raised.

\textit{Suppose $u$ is a solution to the equation
\begin{equation}
\label{eq:H}
H(u):=\Delta u -q(x)u=0,    
\end{equation}
where $\Delta$ is the Laplace operator in $\R^n$ and $q(x)$ is a bounded measurable function in the domain 
$$
\Omega_{\rho}=\{x:|x|>\rho>0\}\subset \R^n, \quad |x|=(x_1^2+\cdots + x_n^2)^{1/2},
$$ 
 such that $|q(x)|\le k^2$, with $k>0$. Moreover, assume that the solution $u$ in $\Omega_{\rho}$ admits the estimate
\begin{equation}
\label{eq:equation}
|u(x)|\le C\exp[-(k+\varepsilon)|x|],
\end{equation}
where $C$ and $\varepsilon$ are positive constants. Do these conditions imply that $u\equiv 0$ in $\Omega_{\rho}$?}

This question is known in the literature as Kondrat'ev--Landis' conjecture, or just Landis' conjecture.
Along the manuscript we will use the terminology \textit{Landis' conjecture}. A weaker version of Landis' conjecture states that if $|u(x)|$ tends to $0$ faster than exponentially at infinity, i.e.,
$$
|u(x)|\le C\exp[-|x|^{1+\varepsilon}], \quad \varepsilon>0, \quad \text{ then } u\equiv0.
$$
This property can be viewed as a \textit{unique continuation at infinity (UCI)} property for solutions satisfying a suitable exponential decay. The question is motivated by the observation that the conjecture is sharp in dimension one, since for a potential $q$ which equals to a constant $c\in\R$ outside a compact set, problem \eqref{eq:H} admits an exponential decaying solution  only if $c=k^2>0$ and, in this case, such a solution decays as $\exp[-|k||x|]$. 
On the other hand, there is an example of a solution to \eqref{eq:H} in $\R^2$ with bounded $q$ that decays exponentially: define $u=\exp[-|x|]$ in $\{|x|>1\}$ and extend it to a $C^2$ smooth positive function on the plane. Then $|\Delta u|\le C|u|$ and, by taking $u\big(\frac{1}{\sqrt{C}}\cdot\big)$ in place of $u$, one can make $|q|\le 1$ in this example. This also motivates the question whether one can construct a nontrivial solution which decays faster than this rate.

A very brief summary of the fundamental advances around the problem are as follows: Landis' conjecture was disproved by V. Z. Meshkov \cite{Me91} who constructed a $q$ and a nontrivial $u$ satisfying $|u(x)|\le C\exp[-C|x|^{4/3}]$. He also showed that if $|u(x)|\le C\exp[-C|x|^{\frac43+}]$\footnote{The notation $|u(x)|\le C\exp[-C|x|^{a+}]$ means that the $|u(x)|$ is decaying as a function $v(x)$ whose decay is faster than $C\exp[-C|x|^{a}]$, namely $\frac{|v(x)|}{C\exp[-C|x|^{a}]}\to 0$ as $|x|\to \infty$. 
}, then $u\equiv 0$.
A quantitative form of Meshkov's result was derived by J. Bourgain and C. E. Kenig in their resolution of Anderson localization for the Bernoulli model \cite{BK05} in higher dimensions. It is worth noting that both $q$ and $u$ constructed by Meskhov are \textit{complex-valued} functions, and actually, for real-valued potentials and solutions, Bourgain--Kenig's quantitative form is an improvement of Meshkov's result. Later on, Kenig, L. Silvestre and J.-N. Wang \cite{KSW15} confirmed Landis' conjecture for any \textit{real} solution $u$ of $\Delta u-\nabla (W)u-qu=0$ in $\R^2$, where $W(x,y)=(W_1(x,y),W_2(x,y))$ and $q(x,y)$ are measurable, \textit{real-valued}, and $q(x,y)\ge 0$ a.e.. Recently, A. Logunov, E. Malinnikova, N. Nadirashvili, and F. Nazarov \cite{LMNN20} confirmed weak Landis' conjecture version for real-valued $q$ in dimension two.  

The aim of this paper is to survey Kondrat'ev--Landis' conjecture, giving an outline of the techniques and ideas of the main advances contained in the breakthroughs above. The current manuscript can be also seen as an extended elaboration of the nice overview by Eugenia Malinnikova \cite{Ma23}. Except for the proof of Lemma \ref{lem:aingeru}, which is due to the second author and Luis Vega, the proofs in this survey are taken from the original articles and presented in a sketched way.  We also complement the article with references and recent contributions directly related to the conjecture, but also to formulations of the problem in other contexts, mentioning further techniques and problems which are, sometimes surprisingly, closely connected with Landis' conjecture. The study of Landis' conjecture has been attracting attention of experts and there is a wide and interesting literature on the topic which we try to collect here. 

\medskip
\noindent{\textbf{Acknowledgments}.}
The authors would like to thank the anonymous referees for an expert
reading and for providing several helpful and detailed comments. 
We are greatly indebted to Biagio Cassano, Ujjal Das, Sylvain Ervedoza, Nicol\`o De Ponti and Eugenia Malinnikova for very helpful discussions and suggestions on the manuscript. After the writing of this manuscript, Fedor Nazarov
sent us a draft \cite{LN24} where  the optimality of \cite{LMNN20} is shown, thus disproving (strong) Landis' conjecture in the plane.

This work was supported by the Research Institute for Mathematical Sciences, an International Joint Usage/Research Center located in Kyoto University.
The second author wishes to thank Neal Bez and Yutaka Terasawa for the organization of the RIMS Symposium on \emph{Harmonic Analysis and Nonlinear Partial Differential Equations} in June 2024, which motivated the writing of this survey.

A. Fern\'andez-Bertolin is supported by the Basque Government through the project IT1615-22 and by the Spanish Agencia Estatal de Investigaci\'on through the project PID2021-122156NB-I00 funded by MICIU/AEI/10.13039/501100011033.

L. Roncal is supported 
by the Basque Government through the BERC 2022--2025 program 
and 
by the Spanish Agencia Estatal de Investigaci\'{o}n
through BCAM Severo Ochoa excellence accreditation CEX2021-001142-S/MCIN/AEI/10.13039/501100011033,
RYC2018-025477-I, CNS2023-143893 funded by MICIU/AEI/10.13039/501100011033 and by the European Union NextGenerationEU/PRTR, and PID2023-146646NB-I00 funded by MICIU/AEI/10.13039/501100011033 and by ESF+, and IKERBASQUE.

D. Stan is supported by the Spanish Agencia Estatal de Investigaci\'{o}n through the project PID2020-114593GA-I00 funded by MICIU/AEI/10.13039/501100011033.

\section{Meshkov's counterexample}
\label{Meshkov}

Shortly after \cite{KL88}, Viktor Zakharovich Meshkov \cite{Me91} investigated the question of the existence of solutions to second-order partial differential equations with decay faster than exponential at infinity. The main \textit{qualitative unique continuation} result, \cite[Theorem 1]{Me91}, states the following. 
\begin{theorem}[{\cite[Theorem 1]{Me91}}]
\label{thm:Me91}
If $u(x)\in H^2_{\operatorname{loc}}(\Omega)$ satisfies the equation \eqref{eq:H} in $\Omega_\rho$ and 
$$
\int_{\Omega_p}|u|^2\exp(2\tau|x|^{4/3})\,dx<\infty,
$$
for all $\tau>0$, then $u\equiv 0$ in $\Omega_{\rho}$.
\end{theorem}
\begin{proof}[Sketch of proof]
The proof is carried out on the basis of a Carleman inequality \cite[Lemma 1]{Me91}. Let $(r,\theta)$ be the polar coordinates in the Euclidean space $\R^n$, $r=|x|=(x_1^2+\cdots + x_n^2)^{1/2}$ and $\theta=x/|x|\in \mathbb{S}^{n-1}$. Denote by $d\theta$ the area element of the unit sphere $\mathbb{S}^{n-1}\subset \R^n$.
Let $2/3<\alpha\le 2$. Then there exist a number $\tau_0>0$ and a constant $C=C_{\alpha}$ such that, for all $C_0^{\infty}$-functions $v$ with support in the set $\Omega_1=\{x\in \R^n: |x|>1\}$ and any $\tau>\tau_0$,
$$
    \tau^3\int |v|^2r^{3\alpha-3}\exp[2\tau r^{\alpha}]\,dr\,d\theta\le C \int |\Delta v|^2 r\exp[2\tau r^{\alpha}]\,dr\,d\theta.
$$
Note that, for $\alpha=4/3$ and sufficiently large $\tau$, in view of the boundedness of $q$, under the conditions above on the function $v$ we obtain  
\begin{equation}
\label{eq:carlemanMeshkov}
 \frac{\tau^3}{2}\int |v|^2r\exp[2\tau r^{4/3}]\,dr\,d\theta\le C \int |(\Delta-q) v|^2 r\exp[2\tau r^{4/3}]\,dr\,d\theta.
\end{equation}
Applying a cut-off $C^{\infty}$-function $h$ which vanishes in some neighborhood of the origin and is equal to one in a neighborhood of infinity, it is easy to verify that $v=hu$ verifies the Carleman inequality \eqref{eq:carlemanMeshkov}. Further absorption arguments and interior Schauder estimates are applied to conclude the proof.
\end{proof}

In the same work, Meshkov presents a construction in the plane $\R^2$ of an example of equation \eqref{eq:H} that has a nonzero solution $u$, for which the estimate $|u(x)|\le C\exp(-c|x|^{4/3})$, for all $x\in \R^2$, is realized with some positive constants $c$ and $C$. This example gives a negative answer to Landis' conjecture. As in the Introduction, we want to emphasize that in Meshkov's counterexample both $q$ and $u$ are \textit{complex-valued} functions. A generalization of Meshkov's example was shown by J. Cruz-Sampedro in \cite{Cruz99}, and consequent follow-ups addressing missing cases of constructions of solutions in $\R^2$ that satisfy an elliptic eigenvalue equation and have the optimal rate of decay at infinity are presented in the more recent works by B. Davey \cite{Davey14, Davey15}. 
We also refer to \cite{DZZ08}, where T. Duyckaerts, X. Zhang, and E. Zuazua observed that an easier construction of examples with critical decay can be done considering \textit{vector-valued} solutions to the equation \eqref{eq:CS} below: a $\C^4$-valued solution $u$ is constructed for \eqref{eq:CS} in $\R^3$ so that $|u(x)|\le \exp[-C|x|^{4/3}]$ for some $C>0$ and for all $x\in \R^3$. In \cite{DZZ08}, the potential constructed is non-symmetric (in the same spirit as Meshkov) and has a logarithmic growth at infinity: so it is not bounded, but it is almost bounded when measuring with respect to power like growth. 

In order to give some insight, we state and sketch here the result and construction by Cruz-Sampedro\footnote{We have kept the original formulation by Cruz-Sampedro, which is slightly different to Kondrat'ev--Landis' conjecture referred to the equation \eqref{eq:H}. From the thesis that $u$ has compact support, using the uniqueness of solutions to the Schr"odinger equation (see, for example, \cite[Theorem 13.63]{RS78}), we find that $u\equiv0$.}, which follow the same ideas as in \cite{Me91}.

\begin{thm}[{\cite[Theorem 1]{Cruz99}}]
\label{thm:Cruz99}
Let $\varepsilon\in\R$ be given and suppose $q$ is a measurable complex-valued function on $\R^n$ that satisfies 
\begin{equation}
    \label{eq:potential}
q(x)=O(|x|^{-\varepsilon})\quad \text{ as } \quad|x|\to\infty.
\end{equation}
Let $3\delta=\max\{0,1-2\varepsilon\}$. Let  $u(x)\in H^2_{\operatorname{loc}}(\R^n)$ be a nonzero solution to
\begin{equation}
\label{eq:CS}
-\Delta u+q u=E u,\qquad E\in \R
\end{equation}
that satisfies
$$
\exp(\tau|x|^{1+\delta})u\in L^2(\R^n)
$$
for all $\tau>0$. Then $u$ has compact support.
\end{thm}
We emphasize that this result shows that, if we allow a faster decay on the potential $q$, then we can prove Landis' conjecture with a slower decay assumption on $u$. On the other hand, if we allow the potential to grow, then we need to assume faster decay of the solution. 
Theorem \ref{thm:Cruz99}, also proven through a \textit{Carleman-type estimate}, generalizes \cite[Theorem 1]{Me91}, which corresponds to taking $\varepsilon =0$ (and hence $\delta=1/3$).  Similar results, in restricted ranges of $\varepsilon$, were previously obtained by Froese, Herbst, T. Hoffmann-Ostenhof, and M. Hoffman-Ostenhof  \cite{FHHH82,FHHH83}.
Cruz-Sampedro's generalization of Meshkov's example is as follows.
\begin{thm}[{\cite[Theorem 2]{Cruz99}}]
\label{thm:Cruz99_2}
Let $\varepsilon<1/2$ and $\delta>0$ satisfy $2\varepsilon+3\delta =1$. Then there exist a continuous complex-valued function $q$ on $\R^2$ satisfying \eqref{eq:potential} and a $C^2$-function $u$ which does not have compact support and satisfies $\Delta u -q(x)u=0$ on $\R^2$ and 
\begin{equation}
\label{eq:decay}
u(x)=O\big(\exp(-\beta|x|^{1+\delta})\big)\quad \text{ as } \quad |x|\to \infty, \text{ for some } \beta>0.
\end{equation}
\end{thm}
Theorem \ref{thm:Cruz99_2} means that, for complex-valued potentials, Theorem \ref{thm:Cruz99} is optimal for all $\varepsilon\in \R$. The general strategy used was that of Meshkov, but Cruz-Sampedro also took strong inspiration from \cite{SS83}.

\subsection{The description of the construction: Cruz-Sampedro's generalization}

To give the construction of the example in Theorem \ref{thm:Cruz99_2}, we first construct solutions on annular regions. The annular constructions are described in Lemma \ref{lem:cruz}. With Lemma \ref{lem:cruz} at hand, the proof of Theorem \ref{thm:Cruz99_2} consists of showing that the solutions on annuli can be put together to give solutions over all of $\R^2$ with the appropriate decay.  
As before, we use polar coordinates $r,\theta$ in the plane $\R^2$. For $\rho>0$ we denote by $A(\alpha,\beta)$ the annulus in $\R^2$ defined by $\rho+\alpha \rho^{(1-\delta)/2}\le r\le \rho +\beta\rho^{(1-\delta)/2}$.

\begin{lem}[{\cite[Lemma 3.1]{Cruz99}}]
\label{lem:cruz}
Let $\delta>0$ and $\varepsilon\in \R$ satisfy $2\varepsilon+3\delta=1$. For a fixed and large $\rho>0$ let $n$ and $k$ be positive integers such that $|n-\rho^{1+\delta}|\le 1$ and $|k-6(1+\delta)\rho^{(1+\delta)/2}|\le 1+20\delta (1+\delta)$. Then there exist complex-valued functions $u$ and $q$ on $A(0,6)$ possessing the following properties (see Figure \ref{fig:annili}):
\begin{enumerate}
\item[(a)] The function $u$ is of class $C^2$ and satisfies
$$
\Delta u+qu=0\quad \text{ on } \quad A(0,6).
$$
\item[(b)] There exists a constant $C$ independent of $\rho, n$, and $k$ such that 
\begin{equation}
\label{eq:decay_potential}
|q(r,\theta)|\le Cr^{-\varepsilon}\quad \text{ on } \quad A(0,6).
\end{equation}
\item[(c)] For a constant $a>0$ we have 
$$
u(r,\theta)=\begin{cases}
r^{-n}\exp(-in\theta)\quad &\text{ on } \quad A(0,0.1),\\
ar^{-n-k}\exp[-i(n+k)\theta]\quad &\text{ on } \quad A(5.9,6).
\end{cases}
$$
From this it follows that $q\equiv 0$ on the annuli $A(0,0.1)$ and $A(5.9,6)$.
\item[(d)] Let $m(r)=\max\{|u(r,\theta)|, 0\le \theta\le 2\pi\}$. Then
$$
\log m(r)-\log m(\rho)\le \log 2-\frac16\int_{\rho}^rt^{\delta}\,dt, \quad \text{ for } \quad \rho\le r\le \rho+6\rho^{(1-\delta)/2}.
$$
\end{enumerate}
\end{lem}
One motivation behind considering the functions in Lemma \ref{lem:cruz} can be seen from the polar form of the Laplacian and the corresponding form of the solutions.
\begin{figure}
\includegraphics[scale=0.3]{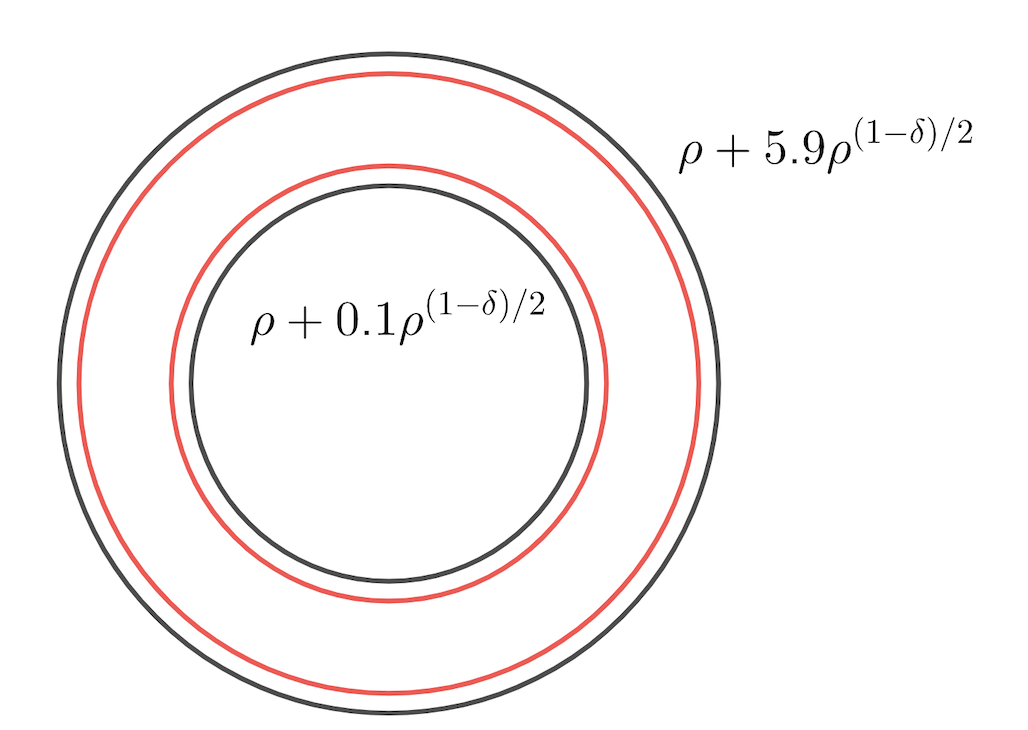}
\caption{Cruz-Sampedro's generalization of Meshkov's construction, Lemma \ref{lem:cruz} (c): $u(r,\theta)=r^{-n}\exp(-in\theta)$ on the inner thin annulus and $u(r,\theta)=ar^{-n-k}\exp[-i(n+k)\theta]$ on the outer thin annulus;
$|q(r,\theta)|\le Cr^{-\varepsilon}$ on the whole annulus and $q\equiv 0$  on the thin annili.
}
  \label{fig:annili}
  \end{figure}
\subsubsection{Sketch of the proof of Theorem \ref{thm:Cruz99_2}}
We use Lemma \ref{lem:cruz} to construct the example that proves Theorem~\ref{thm:Cruz99_2}. We recursively define a sequence of numbers $\{\rho_j\}_{j=1}^{\infty}$ and for $\rho_1$ we choose a sufficiently large positive number. Then, if $\rho_j$ has been chosen, we set $\rho_{j+1}=\rho_j+6\rho_j^{(1-\delta)/2}$. Then we define
$n_j:=[\rho_j^{1+\delta}]$, where $[x]=\max\{n\in \Z: n\le x\}$ and $k_j:=n_{j+1}-n_j$. It can be proved that, if $\rho_1$ is large, we may assume that $|k_j-6(1+\delta)\rho_j^{(1+\delta)/2}|\le 1+ 20\delta (1+\delta)$. 

For $j=1,2,\ldots$, let $a_j$ be constants, and let $u_j$ and $q_j$ be functions constructed on $\rho_j\le r\le \rho_{j+1}$ as in Lemma \ref{lem:cruz}. The decay estimate required for potentials is given by (b) from Lemma \ref{lem:cruz}. Part (c) of Lemma \ref{lem:cruz} shows that $u_j(\rho_j,\theta)=\rho_j^{-n_j}\exp[-in_j\theta]$ and $u_j(\rho_{j+1},\theta)=a_j\rho^{-n_{j+1}}_{j+1}\exp[-in_{j+1}\theta]$. Since $\rho_j\to\infty$ as $j\to \infty$, then for $r>\rho_1$ we set 
$$
q(r,\theta)=q_j(r,\theta)\quad \text{ and } \quad u(r,\theta)=A_ju_j(r,\theta)\quad  \text{ for } \rho_j\le r\le \rho_{j+1}, 
$$
where we define $A_j=a_0\cdot a_1\cdot\ldots\cdots a_{j-1}$, for $j=1,2,\ldots$, and $a_0=1$. Then $q$ satisfies \eqref{eq:potential} and $u$ is of class $C^2$ and satisfies $-\Delta u+qu=0$ on $\Omega_{\rho_1}$. 

To prove that $u$ satisfies \eqref{eq:decay}, we set $m(r)=\max\{|u(r,\theta)|, 0\le \theta\le 2\pi\}$ for $r>\rho_1$ and pick $\ell$ so that $\rho_{\ell}\le r\le \rho_{\ell+1}$. Then
$$
\log m(r)=\big(\log m_{\ell}(r)-\log m_{\ell}(\rho_{\ell})\big)+\cdots+\big(\log m_{1}(\rho_2)-\log m_{1}(\rho_{1})\big)+\log m_{1}(\rho_1),
$$
where $m_j(r)=\max\{|u_j(r,\theta)|, 0\le \theta\le 2\pi\}$ is as in (d) of Lemma \ref{lem:cruz}, noting that $a_j=\frac{m_{j}(\rho_{j+1})}{m_{j+1}(\rho_{j+1})}$. Using (d) we find that
$$
\log m(r)\le \ell \log 2-\frac16\int_{\rho_1}^rt^{\delta}\,dt+\log m(\rho_1).
$$
Thus if $\delta>1$ then $\log m(r)\le Cr^{(1+\delta)/2}-cr^{1+\delta}$, and if $0<\delta\le 1$ then $\log m(r)\le Cr-cr^{1+\delta}$, where $C$ and $c$ are positive constants. Since $\delta>0$, for $r$ sufficiently large we have $
0<m(r)\le C\exp[-\beta r^{1+\delta}]$,
for some $\beta>0$. Finally, $q$ and $u$ defined above can be extended to $\R^2$ in a way that Theorem~\ref{thm:Cruz99_2} is satisfied.

\subsubsection{Sketch of the proof of Lemma \ref{lem:cruz}} We present a brief description of the main points in the proof. For a complete and detailed proof, see \cite[Lemma]{Me91}, \cite[Lemma 3.1]{Cruz99}, \cite[Lemma 6.1]{Davey14}, or \cite[Lemma 2.1]{Davey15}. The idea in the construction is that, as $r$ increases from $\rho$ to $\rho +6\rho^{(1-\delta)/2}$, we smoothly modify in four steps the function $u_1=r^{-n}\exp[-in\theta]$ into a function $u$ in $A(0,6)$ that satisfies (a), (b), (c), (d). 

\noindent \textit{Step 1.} The annulus $A(0,2)$. In a first step, we consider the function 
$$
u_1=r^{-n}\exp[-in\theta], 
$$ 
which is rearranged into a function of the form 
$$
u_2=-br^{-n+2k}\exp[iF(\theta)],
$$ 
where $b=(\rho+\rho^{(1-\delta)/2})^{-2k}$ and $F$ will be defined in a moment. For $m=0,1,\ldots,2n+2k-1$ we set $\theta_m=mT$, where 
$
T:=\pi/(n+k)$.
Let $f$ be a smooth $T$-periodic function on $\R$ such that $\int_0^Tf(\theta)\,d\theta=0$, $f(\theta)=-4k$ on $[0,T/5]\cup [4T/5,T]$, and $-4k\le f(\theta)\le 5k$ and $|f'(\theta)|\le Ck/T$, for $0\le \theta\le T$. Set 
$$
\Phi(\theta)=\int_0^{\theta}f(t)\,dt.
$$
Observe that, for $\theta\in \R$, we have
\begin{equation}
\label{eq:Phim}
\Phi(\theta)=-4k(\theta-\theta_m)=:-4k\theta+b_m \quad \text{ for } \quad|\theta-\theta_m|\le T/5.
\end{equation}
We define 
$
F(\theta)=(n+2k)\theta+ \Phi(\theta)$.
Note that $|u_1(r,\theta)|=|u_2(r,\theta)|$ for $r=\rho+\rho^{(1-\delta)/2}$. Moreover, it follows from \eqref{eq:Phim} that $u_2=-br^{-(n-2k)}\exp[i(n-2k)\theta+ib_m]$ on the sectors
$$
S_m:=\{(r,\theta):|\theta-\theta_m|\le T/5\}, \quad m=0,1,\ldots, 2n+2k-1.
$$

We now choose $C^{\infty}$-functions $\Psi_i(r)$, $i=1,2$, taking values between $0$ and $1$ such that $\Psi_1$ vanishes for $r\ge \rho+1.9\rho^{(1-\delta)/2}$ and equals $1$ for $r\le \rho +(5/3)\rho^{(1-\delta)/2}$, and $\Psi_2$ vanishes for $r\le \rho+0.1\rho^{(1-\delta)/2}$ and equals $1$ for $r\ge \rho +(1/3)\rho^{(1-\rho)/2}$, and moreover
\begin{equation}
\label{eq:Psip}
|\Psi_i^{(p)}(r)|\le C\rho^{-p(1-\delta)/2}, \quad r\ge 0, \quad i=1,2; \,\, p=1,2.
\end{equation}
Define $u=\Psi_1u_1+\Psi_2u_2$. It can be checked that $u$ is harmonic in $S:=A(1/3,5/3)\cap (\cup S_m)$. Set
$$
q(r,\theta)=\begin{cases}
    0\quad &\text{ if } (r,\theta)\in S,\\
    \Delta u/u \quad & \text{ otherwise}.
\end{cases}
$$
Then (a) holds on $A(0,2)$. The proof that $|u|>0$ on $A(0,2)\setminus S$ and that (b) holds on $A(0,2)$ is more involved, and here is where the fact that $q$ is complex-valued is relevant. In particular, the critical point is to show that $|u|>0$ and then (b) holds on 
$$
P_m=\big\{(r,\theta): \theta_m+\frac{T}{5}\le \theta\le \theta_m+\frac{4T}{5}\big\}\cap A(1/3,5/3), \quad m=0,\ldots, 2n+2k-1.
$$
For that purpose, we set 
$$
G(\theta)=F(\theta)+n\theta.
$$
On the annular sectors $P_m$ we have
\begin{equation}
\label{eq:u1masu2}
|u|=|u_1+u_2|=|u_2|\Big|\exp[iG(\theta)]-\frac{1}{br^{2k}}\Big|
\end{equation}
and it will be shown later that for some $\eta>0$
\begin{equation}
\label{eq:G}
\big|\exp[iG(\theta)]-\frac{1}{br^{2k}}\big|\ge \eta, \quad (r,\theta)\in P_m, \quad m=0,\ldots,2n+2k-1.
\end{equation}
On the other hand, by the assumptions on $k,n,\varepsilon$, and $\delta$, we see that
\begin{equation}
\label{eq:deltau2}
|\Delta u_2|\le Cr^{-\varepsilon}|u_2|\quad \text{ on } A(0,2).
\end{equation}
With \eqref{eq:G} and \eqref{eq:deltau2} and the fact that $\Delta u=\Delta u_2$ on $P_m$ we obtain $|u|>\eta|u_2|$ and so (b) holds on $P_m$. 

Let us show \eqref{eq:G}. Observe that $G(\theta)=2(n+k)\theta+\Phi(\theta)$ and $G'(\theta)=2(n+k)+f(\theta)$. By the assumptions of $f,k$, and $n$, we may assume that $G'(\theta)>n>0$. Since $G(\theta_m)=2\pi m$ and $G(\theta_{m+1})=2\pi (m+1)$ we conclude that 
$$
2\pi m+ \frac{nT}{5}\le G(\theta)\le 2\pi (m+1)-\frac{nT}{5}\quad \text{ for } \quad \theta_m+\frac{T}{5}\le \theta \le \theta_m+\frac{4T}{5}.
$$
Using the definition of $T$ and the assumptions on $k$ and $n$ one finds that
$$
2\pi m+ \frac{\pi}{7}\le G(\theta)\le 2\pi (m+1)-\frac{\pi}{7}\quad \text{ for } \quad \theta_m+\frac{T}{5}\le \theta \le \theta_m+\frac{4T}{5}.
$$
From this we conclude that \eqref{eq:G} holds with $\eta=\sin(\pi/7)$. We emphasize that it is crucial  that the bound in  \eqref{eq:G} is obtained taking the imaginary part of $\exp[iG(\theta)]$ into account and this is possible by the complex nature of the constructed function $u$ (and the corresponding complex-valued potential). 

\noindent \textit{Step 2.} On $A(2,3)$, we deform $u_2=-br^{-n+2k}\exp[iF(\theta)]$ into $u_3=-br^{-n+2k}\exp[i(n+2k)\theta]$. Let $\Psi(r)$ be a $C^{\infty}$-function which takes values between $0$ and $1$, equals $1$ for $r\le \rho+(7/3)\rho^{(1-\delta)/2}$, vanishes for $r\ge \rho+(8/3)\rho^{(1-\delta)/2}$, and satisfies \eqref{eq:Psip}. On $A(2,3)$ we set $u=-br^{-n+2k}\exp[i\big(\Psi(r)\Phi(\theta)+(n+2k)\theta\big)]$ and $q=\Delta u/u$. It can be shown that (b) holds on $A(2,3)$.

\noindent \textit{Step 3.} On $A(3,4)$, we deform $u_3$ into $u_4=-bdr^{-(n+2k)}\exp[i(n+2k)\theta]$, where $b$ is as in Step 1 and $d:=(\rho +3\rho^{(1-\delta)/2})^{4k}$. Let $\Psi$ be a $C^{\infty}$-function that takes values between $0$ and $1$, equals $1$ for $r\le \rho+(10/3)\rho^{(1-\delta)/2}$, vanishes for $r\ge \rho+(11/3)\rho^{(1-\delta)/2}$, and satisfies \eqref{eq:Psip}. Next, we define $h(r)=\Psi(r)+\big(1-\Psi(r)\big)dr^{-4k}$. It can be verified that $h$ satisfies \eqref{eq:Psip}. Now we set $u=u_3h$ and $q=\Delta u/u$ and it can be verified that (b) holds on $A(3,4)$. Furthermore, on $A(11/3,4)$ we have $u=-bdr^{-(n+2k)}\exp[i(n+2k)\theta]$.

\noindent \textit{Step 4.} Finally, on $A(4,6)$ one deforms $u_4$ into $u_5=ar^{-n-k}\exp[i(-n-k)\theta]$, where $a=bd(\rho+5\rho^{(1-\delta)/2})^{-k}$ and $b$ and $d$ are as in Step 3. Here, $a$ is chosen so that $|u_4(r,\theta)|=|u_5(r,\theta)|$ for $r=\rho+5\rho^{(1-\delta)/2}$. Let $\psi_i(r)$, $i=1,2$, be $C^{\infty}$-functions taking values between $0$ and $1$ satisfying \eqref{eq:Psip} such that $\psi_1$ vanishes for $r\ge \rho+5.9\rho^{(1-\delta)/2}$ and equals $1$ for $r\le \rho+(17/3)\rho^{(1-\delta)/2}$, and $\psi_2$ vanishes for $r\le \rho+4.1\rho^{(1-\delta)/2}$ and equals $1$ for $r\le \rho+(13/3)\rho^{(1-\delta)/2}$. Now, on $A(4,6)$ we set $u=\psi_1u_4+\psi_2u_5$. It turns out that $u$ is harmonic on $A(13/3,17/3)$. Therefore, we let $q=0$ on this annulus. It is also verified as in Step 1 that $q=\Delta u/u$ satisfies \eqref{eq:decay_potential} on the remaining points of $A(4,6)$.

Now let us continue to finish the proof of Lemma \ref{lem:cruz}. Set $m(r)=\max\{|u(r,\theta)|, 0\le \theta\le 2\rho\}$ and
$$
M(r)=\begin{cases}
    r^{-n}\quad &\rho\le r\le \rho+\rho^{(1-\delta)/2},\\
    br^{-n+2k}\quad &\rho+\rho^{(1-\delta)/2}\le r\le \rho+3\rho^{(1-\delta)/2},\\
     br^{-n+2k}h(r)\quad &\rho+3\rho^{(1-\delta)/2}\le r\le \rho+4\rho^{(1-\delta)/2},\\
      bdr^{-n-2k}\quad &\rho+4\rho^{(1-\delta)/2}\le r\le \rho+5\rho^{(1-\delta)/2},\\
       ar^{-n-k}\quad &\rho+5\rho^{(1-\delta)/2}\le r\le \rho+6\rho^{(1-\delta)/2},
\end{cases}
$$
where $a,b$, and $d$ are as in Step 4. Note that $M(r)$ is equal to the modulus of the functions $u_1,\ldots, u_5$ from which the solution $u(r,\theta)$ is constructed. It can be seen that $M(r)$ is a continuous piecewise smooth function on $[\rho, \rho+6\rho^{(1-\delta)/2}]$ that satisfies $m(r)\le 2M(r)$, $m(\rho)=M(\rho)$, and 
$$
\frac{d}{dr}\log M(r)=\frac{-n+O(k)}{r}\le -\frac{\rho^{1+\delta}}{2r}\le -\frac{1}{6}(\rho+6\rho^{(1-\delta)/2})^{\delta}\le -\frac{1}{6}r^{\delta}.
$$
Therefore, if $\rho\le r\le \rho+6\rho^{(1-\delta)/2}$, then
$$
\log m(r)-\log m(\rho)\le \log 2+\int_{\rho}^r\frac{d}{dr}\log M(t)\,dt\le \log 2-\frac{1}{6}\int_{\rho}^r t^\delta\,dt,
$$
concluding the proof of Lemma \ref{lem:cruz}.

\begin{rmk}
The choice of the exponent $4/3=1+1/3$ in the proof of Theorem \ref{thm:Me91} looks rather arbitrary. It is worth carefully studying \cite[Theorem 1]{Cruz99} and \cite[Lemma 3.1]{Cruz99}, corresponding to Theorem \ref{thm:Cruz99} and Lemma \ref{lem:cruz} above, where Cruz-Sampedro introduces the parameter $\delta$ and the relevant exponents are $1+\delta$. The assumption for 
$\delta$ and $\varepsilon$ in \cite[Lemma 3.1]{Cruz99} is somewhat hidden within the proof, and it leads to the choice $\delta=1/3$ when $\varepsilon=0$. We did not include a full track of this assumption in the current manuscript, but we point out that essentially an equivalent formulation is used by Davey in \cite[Lemma 6.1]{Davey14} or \cite[Lemma 2.1]{Davey15}: for instance, the condition $2\varepsilon+3\delta=1$ in \cite[Lemma 3.1]{Cruz99} is essentially the condition $\beta_0=\frac{4-2N}{3}$ in \cite[Lemma 6.1]{Davey14}. Moreover, it is easier to track how this assumption arises in \cite[Lemma 6.1]{Davey14} or \cite[Lemma 2.1]{Davey15}. On the other hand, one can see that $(\beta_0-1)$ and $N$ in \cite{Davey14,Davey15} play the role of $\delta$ and $\varepsilon$, respectively, in \cite{Cruz99}. The threshold $N=0$ leads to $\beta_0-1=\delta=1/3$.
\end{rmk}

Hence, Meshkov gave a partial answer to the question of the possible rate of decay at infinity of solutions of equations of the form \eqref{eq:H} in the case where only the condition of boundedness is imposed on the potential $q(x)$. 
Nevertheless, if $q(x)$ additionally satisfies the smoothness condition $|\nabla q(x)|=O(|x|^{-1})$ as $|x|\to \infty$, then in this case equation \eqref{eq:H} cannot have a solution decaying superexponentially at infinity. This was proven in a somewhat more general situation in  \cite{Me89}.
In  \cite{Me91}, Meshkov also considered parabolic and hyperbolic equations
on a cylinder $M\times \R_+$, where $M$ is a closed compact $C^{\infty}$-manifold of dimension $n$ with a $C^{\infty}$-density $d\mu$, $P$ a second-order elliptic operator with $C^{\infty}$ coefficients that is self-adjoint and upper semi-bounded in $L^2(M,d\mu)$, and $q(x,t)$ a bounded measurable coefficient.

\section{A quantitative form of Meshkov's result by Bourgain and Kenig}
\label{BK}

In the seminal work \cite{BK05}, Jean Bourgain and Carlos E. Kenig  studied the problem of Anderson localization for the continuous Bernoulli model, which is a well-known problem in the theory of disordered media\footnote{The summary in this section closely follows \cite{K05,K07}.}. The problem of localization originates in a 1958 paper by Phil Anderson \cite{A58}, who studied random Schr\"odinger operators on $\ell^2(\Z^n)$ given by 
\begin{equation}
\label{eq:Hd}
H_{\operatorname{d}}=-\Delta_{\operatorname{d}}+\lambda q,
\end{equation}
where $\Delta_{\operatorname{d}}u(x)=\sum_{|y-x|=1}\big(u(y)-u(x)\big)$ is the discrete Laplacian, $q:\Z^n\to \R$ a random field with i.i.d. components, and a real parameter $\lambda$ representing the noise or disorder strength\footnote{Anderson's original work deals with the Schr\"odinger equation in the $3$-dimensional lattice $i\hslash\frac{du}{dt}=Hu$, where $Hu_j=E_ju_j+\sum_{k\neq j}q_{jk}u_k$ and $j,k$ are lattice locations. Here $E_j$ is the energy of a spin occupying a site $j$ and $q_{jk}$ represents the transfer of a spin from one site to another, and may be a stochastic variable with a probability distribution.}. Anderson's work suggested,  using physical arguments, that for large $\lambda$, with probability $1$, a typical realization of the random operator $H_{\operatorname{d}}$ exhibits exponentially decaying eigenfunctions which form a basis of $\ell^2(\Z^n)$. This is known as \textit{Anderson localization}.

Bourgain and Kenig considered a random Schr\"odinger operator on $\R^n$ of the form
$$
H_{\varepsilon}=-\Delta + q_{\varepsilon},
$$
where $q_{\varepsilon}(x)=\sum_{j\in \Z^n}\varepsilon_j \varphi(x-j)$.
Here, $\varepsilon_j\in \{0,1\}$ are independent and $\varphi$ is a smooth, compactly supported function satisfying $0\le \varphi \le 1$. This is commonly referred to as the (continuous) Anderson--Bernoulli model. 
In this context, Anderson localization means that, near the bottom of the spectrum (i.e., for energies $0<E<\delta$, with $\delta=\delta(n)$ small), $H_{\varepsilon}$ has pure point spectrum with exponentially decaying eigenfunctions, almost surely (a.s.). This phenomenon was well-understood in the case when the random potential $q_{\varepsilon}$ has a continuous site distribution (i.e., the $\varepsilon_j$ take their values in $[0,1]$). When the random variables $\varepsilon_j$ are discrete valued (that is, the Anderson--Bernoulli model), the result was established for $n=1$ by Carmona--Klein--Martinelli \cite{CKM87} and by Shubin--Vakilian--Wolff \cite{SVW88}, using different methods (the Furstenberg--Lepage approach and the supersymmetric formalism, respectively). Neither of these methods was extended to $n>1$, until the work of Bourgain and Kenig.

The main result in \cite{BK05} is as follows. 
\begin{thm}[{\cite{BK05}}]
\label{thm:BK}
    For energies near the bottom of the spectrum, $0<E<\delta$, $H_{\varepsilon}$ displays Anderson localization a.s. in $\varepsilon$ for $n\ge1$. 
\end{thm}
The only previous result when $n>1$ was due to Bourgain \cite{B04}, who considered $q_{\varepsilon}$ with $\varphi(x)\sim \exp[-|x|]$, instead of $\varphi\in C_0^{\infty}$. The non-vanishing of $\varphi$ as $|x|\to\infty$ was essential in Bourgain's argument. 
In \cite{BK05}, on the true Bernoulli model, Bourgain and Kenig overcome this problem by the use of a quantitative uniqueness result. The general strategy in the proof of Theorem \ref{thm:BK} is based on a method called ``multi-scale analysis'', or a procedure by an ``induction on scales'' argument,  developed by Fr\"ohlich--Spencer \cite{FS83} and Fr\"ohlich--Martinelli--Scoppola--Spencer \cite{FMSS85}. We refer to \cite{Sc22} for a nice exposition of multiscale techniques in the theory of Anderson localization.
Bourgain and Kenig consider the restriction of the operator to a cube of size $\ell$ (under Dirichlet boundary conditions) and establish their estimates by induction in $\ell$. Such estimates are weak versions of the so-called ``Wegner estimates'' \cite{W81}, which roughly show that, for a large set of $\varepsilon$, at scale $\ell$, one has ``good resolvent estimates'' favorably depending on $\ell$. In proving such an estimate in the Bernoulli case, one of the key tools is a probabilistic lemma on Boolean functions, already used by Bourgain \cite{B04}.

\begin{lem}[{\cite[Lemma 3.1]{BK05}}]
\label{lem:BK}
Let $f=f(\varepsilon_1,\dots,\varepsilon_n)$ be a bounded function on $\{0,1\}^n$ and denote $I_j=f|_{\varepsilon_j=1}-f|_{\varepsilon_j=0}$, the $j^{\text{th}}$ influence, which is a function of $\varepsilon_{j'}$, $j'\neq j$. Let $J\subset \{1,\ldots,n\}$ be a subset with $|J|\le \delta^{-1/4}$, so that $0<k<|I_j|<\delta<1$ for all $j\in J$. Then, for all $E$, 
$$
\operatorname{meas} \{|f-E|<k/4\}\le |J|^{-1/2},
$$
where $\operatorname{meas}$ refers to the normalised counting measure on $\{0,1\}^n$.
\end{lem}
The function to which this lemma is applied to is the eigenvalue $E$, i.e., $H_{\varepsilon}\xi=E\xi$. It then becomes critical to find bounds for the $j^{\text{th}}$ influence of eigenvalues. One can see that if $\xi$ is a normalised eigenstate ($\|\xi\|_{L^2}=1$) with eigenvalue $E$, by first order eigenvalue variation, one has that $I_j=\int|\xi(x)|^2\varphi(x-j)\,dx$. Upper bounds for this are more or less standard and what is at issue are  lower bounds for $I_j$. This can be reformulated to the following quantitative unique continuation problem at infinity:

Suppose that $u$ is a solution to 
$$
\Delta u+qu=0\quad \text{ in } \R^n, \quad \text{ with } \|q\|_{\infty}\le 1,
$$
so that $\|u\|_{\infty}\le C_0$ and $u(0)=1$. Carleman's unique continuation principle \cite{H83} dictates that $u$ cannot vanish identically on any open set, that is, for each $x_0\in \R^n$, $\sup_{x\in B(x_0,1)}|u(x)|>0$. Bourgain and Kenig derived a quantitative version of this property. More precisely, for $R$ large, define
$$
M(R)=\inf_{|x_0|=R}\sup_{B(x_0,1)}|u(x)|.
$$
The question that  needs to be addressed is: 

\begin{center}
\textit{How small can $M(R)$ be?}
\end{center}
The answer by Bourgain and Kenig is contained in the following theorem.
\begin{theorem}[{\cite[Lemma 3.10]{BK05}}]
\label{thm:BKquantitative}
    Under the above conditions on $u$, we have
   \begin{equation}      
\label{eq:Carleman}
    M(R)\ge C\exp[-CR^{4/3}\log R]|u(0)|\quad \text{ for some } C>0 \text{ and } R\to \infty.
    \end{equation}
\end{theorem}
Quantitative estimates for solutions of $-\Delta u+W  \cdot \nabla u+qu=\lambda u$ in the spirit of Theorem \ref{thm:BKquantitative}  were studied in \cite{Davey14}, see also the generalizations by  C.-L. Lin and J.-N. Wang \cite{LW14}.
It is interesting to note that the corresponding Theorem \ref{thm:BKquantitative} is false on the lattice $\Z^n$, see, for instance, the article by Jitomirskaya \cite[Theorem 2]{J07}. As a remark, in order for the induction on scales argument to work in proving the weak Wegner estimate, an estimate of the form
$$
 M(R)\ge C\exp[-CR^{\beta}], \quad \text{ with } \beta<\frac{1+\sqrt3}{2}=1.35\ldots, 
$$
was needed. Observe that  $4/3=1.33\ldots$. It turns out that the result in Theorem \ref{thm:BKquantitative} is a quantitative version of Landis' conjecture. On the other hand, Meshkov's example in Theorem \ref{thm:Cruz99_2} clearly shows the sharpness of the lower bound on $M(R)$ in Theorem \ref{thm:BKquantitative}. Nevertheless, in Meshkov's example, $u$ and $q$ are complex valued, while for many applications, only real $u$, $q$ are of interest. Then, the following question arises:
\begin{question}{\cite[Question 1]{K05}, \cite[p. 28]{K07}}
\label{q:Kenig}
    Can $4/3$ in Theorem \ref{thm:BKquantitative} be improved to $1$ for real-valued $u$, $q$?
\end{question}

\subsection{Sketch of the proof of Theorem \ref{thm:BKquantitative}.}

\subsubsection{A Carleman-type inequality.}

We turn to a sketch of the proof of Theorem \ref{thm:BKquantitative}. The starting point is a Carleman-type inequality \cite{H83}. A proof, essentially  due to L. Escauriaza and S. Vessella \cite{EV02}, is given in \cite[Appendix 8]{BK05}. In the current manuscript we will provide still an alternative proof by the first author and Luis Vega.
\begin{lem}{\cite[Lemma 3.15]{BK05}}
\label{lem:aingeru}
There are positive dimensional constants $C_1,C_2,C_3$, and an increasing function $w(r)$ defined for $0<r<10$, so that 
\begin{equation}
\label{eq:equiC1}
\frac{1}{C_1}\le \frac{w(r)}{r}\le C_1
\end{equation}
and such that, for all $f\in C_0^{\infty}(B(0,10)\setminus\{0\})$, $a>C_2$ we have
\begin{equation}
\label{eq:carlemanBilbao}
a^3\int w^{-1-2a}|f|^2\le C_3\int w^{2-2a}|\Delta f|^2.
\end{equation}
\end{lem}

\begin{proof}[Sketch of proof (due to A. Fern\'andez and L. Vega)]

We will construct a radial function (we will write $r=|x|$ in the sequel) $\phi(r)$ such that, for $a$ large enough
$$
a^3\int \frac{e^{-2a\phi}}{|x|^3}|f|^2+a\int \frac{e^{-2a\phi}}{|x|}|\nabla f|^2\le C_3\int e^{-2 a\phi}|\Delta f|^2,
$$
and the construction of $\phi$ will relate this inequality to the one in the statement involving the weight $w$. In order to do this, we write $u=e^{-a \phi}f$ and decompose $e^{-a\phi}\Delta e^{a\phi}=\mathcal{S}+\mathcal{A}$, being $\mathcal{S}$ symmetric and $\mathcal{A}$ antisymmetric, given by (see for instance \cite{EKPV-JEMS}),
\begin{align*}
    \mathcal{S}u&=\Delta u+a^2|\nabla\phi|^2u,\\
    \mathcal{A}u&=a\Delta\phi u+2a\nabla\phi\cdot\nabla u.
\end{align*}

Also, by direct computations
\begin{align*}
[\mathcal{S},\mathcal{A}]u=a[4\nabla\cdot(D^2\phi\nabla u)-4a^2D^2\phi\nabla\phi\cdot\nabla\phi u+\Delta^2\phi u].
\end{align*}

Taking into account that $\phi$ is a radial function and by using integration by parts we have, for general functions $\psi_1$ and $\psi_2$,
\begin{align*}
\langle \mathcal{S}u,\psi_1u\rangle&=\int \Delta u\psi_1 u+a^2\int(\phi')^2\psi_1|u|^2=-\int \psi_1|\nabla u|^2-\int u\nabla u\cdot \nabla \psi_1 +a^2\int(\phi')^2\psi_1|u|^2\\&=-\int \psi_1|\nabla u|^2+\frac12\int \Delta\psi_1|u|^2 +a^2\int(\phi')^2\psi_1|u|^2,\\
\langle \mathcal{A}u,\psi_2u\rangle&=a\int \Delta \phi\psi_2|u|^2+2a\int\nabla\phi\cdot\nabla u\psi_2u=-a\int \phi'\psi_2'|u|^2.
\end{align*}

Furthermore, it is not difficult to check that
\begin{align*}
D^2\phi\nabla\phi\cdot\phi=\phi''(\phi')^2,\ \ D^2\phi\nabla u\cdot\nabla u=\phi''|\partial_r u|^2+\frac{\phi'}{r}|\nabla_\tau f|^2,
\end{align*}
where $\partial_r u=\nabla u\cdot \frac{x}{r}$ and $|\nabla_\tau u|^2=|\nabla u|^2-|\partial_r u|^2$ represent the radial and tangential components of the gradient.

Hence, for general radial functions $\psi_1$ and $\psi_2$, the following holds
\begin{align*}
\| e^{-a\phi} \Delta f\|^2=&\|e^{-a\phi}\Delta e^{a\phi}u\|^2=\|\mathcal{S}u\|^2+\|\mathcal{A}u\|^2+\langle[\mathcal{S},\mathcal{A}]u,u\rangle\\
=&\left\|\Big(\mathcal{S}-\frac{\psi_1}{2}\Big)u\right\|^2+\left\|\Big(\mathcal{A}-\frac{\psi_2}{2}\Big)u\right\|^2-\frac{1}{4}\|\psi_1u\|^2-\frac{1}{4}\|\psi_2u\|^2+\langle\mathcal{S}u,\psi_1u\rangle+\langle\mathcal{A}u,\psi_2u\rangle\\
=&-4a^3\int (\phi')^2\Big(\phi''-\frac{\psi_1}{4a}\Big)|u|^2-4a\int \Big(\phi''+\frac{\psi_1}{4a}\Big)|\partial_r u|^2-4a\int \Big(\frac{\phi'}{r}+\frac{\psi_1}{4a}\Big)|\nabla_\tau u|^2\\&\quad+2a\int \frac{\Delta\psi_1}{4a}|u|^2-4a^2\int \frac{\psi_1^2}{16a^2}|u|^2+a\int\Delta^2\phi|u|^2-a\int\phi'\psi_2'|u|^2\\&\quad-\frac14\int\psi_2^2|u|^2+\left\|\Big(\mathcal{S}-\frac{\psi_1}{2}\Big)u\right\|^2+\left\|\Big(\mathcal{A}-\frac{\psi_2}{2}\Big)u\right\|^2.
\end{align*}

Next, we define $\phi$ according to $\phi'(r)=\frac{e^{-r}}{r}$ and take $\psi_1$ so that $\phi''-\frac{\psi_1}{4a}=-\frac{e^{-r}}{2r}$. Straightforward computations show the existence of dimensional constants $\beta_1,\dots,\beta_5$ such that
\begin{align*}
\| e^{-a\phi} \Delta f\|^2=&\left\|\Big(\mathcal{A}-\frac{\psi_2}{2}\Big)u\right\|^2+\left\|\Big(\mathcal{S}-\frac{\psi_1}{2}\Big)u\right\|^2+2a^3\int \frac{e^{-3r}}{r^3}|u|^2
+6a\int \frac{e^{-r}}{r}|\partial_r u|^2+2a\int \frac{e^{-r}}{r}|\nabla_\tau u|^2\\&+8a\int \frac{e^{-r}}{r^2}|\partial_r u|^2-2(n-4)^2a\int \frac{e^{-r}}{r^4}|u|^2+a\int\frac{e^{-r}}{r}\psi_2'|u|^2-\frac14\int\psi_2^2|u|^2-4a^2\int \frac{e^{-2r}}{r^4}|u|^2\\&+a\int \Big[\frac{\beta_1}{r^3}+\frac{\beta_2}{r^2}+\frac{\beta_3}{r}\Big]e^{-r}|u|^2-4a^2\int \Big[\frac{\beta_4}{r^3}+\frac{\beta_5}{r^2}\Big]e^{-2r}|u|^2.
\end{align*}

Before defining the function $\psi_2$, let us consider the symmetric operator $X=\frac{x}{|x|^4}$ and the antisymmetric operator $Y=\nabla,$ whose commutator is $\frac{n-4}{|x|^4}$. Hence,
\begin{align*}
    (n-4)\int \frac{|v|^2}{|x|^4}=\langle [X,Y]v,v\rangle=2\langle Xv,Yv\rangle=2\int \nabla v\cdot\frac{x v}{|x|^4}&=2\int \frac{\frac{x\cdot\nabla v}{|x|}}{|x|}\frac{v}{|x|^2}=2\int \frac{\partial_r v}{|x|}\frac{v}{|x|^2}\\&\le 2\left(\int\frac{|\partial_r v|^2}{|x|^2}\right)^{1/2}\left(\int\frac{|v|^2}{|x|^4}\right)^{1/2}
\end{align*} 
implies the Hardy-type inequality
$$
\frac{(n-4)^2}{4}\int \frac{|v|^2}{r^4}\le \int \frac{|\partial_r v|^2}{r^2}.
$$
If we let $v=e^{-r/2}u$, we observe that     
$$
8a\int \frac{e^{-r}}{r^2}|\partial_r u|^2-2(n-4)^2a\int \frac{e^{-r}}{r^4}|u|^2\ge -4(n-3)a\int \frac{e^{-r}}{r^3}|u|^2.
$$
Besides, taking $\psi_2(r)=-4a \frac{e^{-r}}{r^2}$, we see that
$$
a\int\frac{e^{-r}}{r}\psi_2'|u|^2-\frac14\int\psi_2^2|u|^2-4a^2\int \frac{e^{-2r}}{r^4}|u|^2=4a^2\int \frac{e^{-2r}}{r^3}|u|^2.
$$
Therefore, there exist dimensional constants (abusing notation, we keep the notation $\beta_i$) such that
\begin{align*}
\| e^{-a\phi} \Delta f\|^2\ge&2a^3\int \frac{e^{-3r}}{r^3}|u|^2
+6a\int \frac{e^{-r}}{r}|\partial_r u|^2+2a\int \frac{e^{-r}}{r}|\nabla_\tau u|^2\\&+a\int \Big[\frac{\beta_1}{r^3}+\frac{\beta_2}{r^2}+\frac{\beta_3}{r}\Big]e^{-r}|u|^2-4a^2\int \Big[\frac{\beta_4}{r^3}+\frac{\beta_5}{r^2}\Big]e^{-2r}|u|^2.
\end{align*}

To finish the proof, notice that, undoing the change $u=e^{-a\phi}f$, the same inequality holds true replacing $u, \partial_r u$, and $\nabla_\tau u$ by $e^{-a\phi}f,\ e^{-a\phi}\partial_r f$, and $e^{-a\phi}\nabla_\tau f$ respectively, possibly for different dimensional constants $\beta_i$. Further, the support condition on $f$, for which the term $e^{-r}$ is bounded from below and from above, and the fact that for $a$ large enough the integrals in the second line of the previous formula can be absorbed in a fraction of the first line, imply that
$$
\| e^{-a\phi} \Delta f\|^2\gtrsim a^3\int \frac{e^{-2a\phi}}{r^3}|f|^2
+a\int \frac{e^{-2a\phi}}{r}|\nabla f|^2.
$$

Looking at the definition of $\phi$ (i.e. $\phi'(r)=\frac{e^{-r}}{r}$), it is not difficult to construct the function $w(r)$ for which the inequality in \eqref{eq:carlemanBilbao} holds. Indeed, take $w(r) = e^{\frac{a}{a-1}\phi(r)}$. We have that $c_1/r \le \phi'(r) \le c_2/r$,  since $r \in(0,10)$. We also take $a>C_2$ for a large enough positive constant $C_2$. With this, $w(r)$ is increasing on $0<r<10$ and \eqref{eq:equiC1} holds.
Now, taking $w^{2-2a} = e^{-2a \phi}$, we have that $w^{2-2a}/r^3\simeq_{C_1} w^{-1-2a}$, and we are done.
\end{proof}

A classical application of this lemma (see \cite{H83}) is a unique continuation result, which in this case says that if a solution $u$ of \eqref{eq:equation} decays as $o(\exp[-c|x|^{4/3}\log |x|])$, then $u\equiv0$. More generally, this result is due to Carleman \cite{C39}.
 \begin{prop}{\cite{C39}}
 \label{prop:carleman}
     Assume that $\Delta u=qu$ in $B(0,10)$ and that $\|u\|_{L^{\infty}}\le C_0$, $\|q\|_{L^{\infty}}\le M$. Suppose that $|u(x)|\le C_N|x|^N$ for each $N>0$. Then $u\equiv 0$ in $B(0,10)$.

 \end{prop}

 Actually, the power $a^3$ on the left-hand side of the inequality in  Lemma \ref{lem:aingeru} is not relevant for the proof of the corresponding version of Proposition \ref{prop:carleman} and any $h(a)$ with $h(a)\to\infty$ would do. On the other hand, for Theorem \ref{thm:BKquantitative}, the exact power is crucial and as it can be seen from the sketch of the proof below, Meshkov's example implies that no higher power than $3$ can be used, no matter what the choice of $w$ is. 

\subsubsection{Sketch of the proof of Theorem \ref{thm:BKquantitative}.} 
\label{sub:BKquant}

Pick $x_0$ with $|x_0|=R$ such that $M(R)=\sup_{B(x_0,1)}|u(x)|$. We now ``interchange $0$ and $x_0$'' and  ``rescale to $R=1$'' by setting $u_R(x)=u\Big(AR\big(x+\frac{x_0}{AR}\big)\Big)$, where $A$ is a large dimensional constant to be fixed later. We have $|\Delta u_R(x)|\le A^2R^2|u_R(x)|$ and if $\widetilde{x}_0=-x_0/(AR)$, then $u_R(\widetilde{x}_0)=1$ and $|\widetilde{x}_0|=\frac{1}{A}$. Moreover, $M(R)=\sup_{|x|\le 2r_0}|u_R(x)|$, where $r_0=\frac{1}{2AR}$. Pick now a cut-off function $\rho$ satisfying the properties
$$
\begin{cases}
    \rho\equiv 0\quad \text{ on } |x|<\frac{r_0}{2} \text{ and } |x|>4\\
    \rho\equiv 1\quad \text{ on } \frac{2}{3}r_0<|x|<3
\end{cases}
$$
and apply Lemma \ref{lem:aingeru} to $f=(u_R\rho)$. We obtain
\begin{align*}
    a^3\int w^{-1-2a}|u_R|^2\le C_3\int w^{2-2a}\rho^2|\Delta u_R|^2&+C_3\Big[\int_{\frac12r_0<|x|<\frac23r_0}w^{2-2a}\big(|\Delta \rho|^2|u_R^2|+2|\nabla \rho|^2|\nabla u_R|^2\big)\Big]\\&+C_3\Big[\int_{3<|x|<4}w^{2-2a}\big(|\Delta \rho|^2|u_R^2|+2|\nabla \rho|^2|\nabla u_R|^2\big)\Big].
\end{align*}
Observe that $|\Delta u_R(x)|^2\le A^4R^4|u_R(x)|^2$, hence taking
$$
a\simeq R^{4/3} 
$$
we can absorb the first term on the right-hand side into the left-hand side. The left-hand side can be seen to be greater than or equal to $
Ca^3A^{-n}R^{-n}w^{-1-2a}(2/A)$, using that $u_R(\widetilde{x}_0)=1$. The last two terms of the right-hand side are bounded from above by $(CR)^{2a-n+2}M(R)^2$ and by $CC_0^2A^2R^2w(3)^{-2-2a}$, respectively, using interior estimates to get the former. Thus, taking $A$ large enough so that $w(2/A)\le \frac{1}{10}w(3)$ and $R$ large, depending on $n, C_0$, we obtain
$$
Ca^3R^{-n}w^{-1-2a}\Big(\frac2A\Big)\le (CR)^{2a+2-n}M(R)^2,
$$
which, since $a^3\simeq R^4$, gives the desired lower bound and hence Theorem \ref{thm:BKquantitative} is proven.
 
 Note that this improves Meshkov's result in the real-valued case by allowing a slower decay of $u$ than that of Meshkov (which is precisely $\exp[-|x|^{4/3+\varepsilon}]$) for the validity of Landis' conjecture; yet this decay is faster than what Landis originally proposed. On the other hand, Meshkov's example clearly shows the limitations of \textit{Carleman's approach}, which does not distinguish between the real and complex case.  Similar quantitative unique continuation for higher order elliptic equations have been considered in \cite{Z16}, see also \cite{HWZ16}.
Interestingly, Bourgain and Kenig did not dispose of a discrete version of unique continuation or an inequality of the type \eqref{eq:Carleman}, and the case of the lattice version of the Anderson--Bernoulli model for $n\ge 2$ remained unsettled. Recently, there has been further breakthrough on the problem of localization near the edge for the Anderson--Bernoulli model on the lattice, see Section~\ref{subsec:Ding-Smart}. 

In connection with Landis' conjecture, Davey \cite{DJFA20} studied quantitative unique continuation properties of solutions to elliptic equations with singular lower order terms. Davey established bounds for the order of vanishing and the optimal rate of decay at infinity of solutions to equations of the form $\Delta u+ qu=0$, where $q\in L^t$ for any $t\in (\frac{n}{2}, \infty]$ and $n\ge 3$, see \cite[Theorem 1]{DJFA20}. This order of vanishing estimate, together with the scaling argument in \cite{BK05}, yield to a quantitative Landis-type theorem in this setting.

\section{A weaker version of Landis' conjecture: real-valued potentials}

As is clear in Sections \ref{Meshkov} and \ref{BK}, Carleman's method is a powerful technique in studying questions related to Landis’ conjecture and is able to produce optimal bounds in the
complex-valued case. But since Carleman's estimate does not seem to distinguish
real- or complex-valued functions, a direct use of such estimates to resolve Landis’
conjecture seemed likely to fail.
We explained in the previous section that Carleman's method's limitation led Kenig to ask about the study of a weaker version of Landis' conjecture in the real-valued case, i.e., to examine whether the condition $u=o(\exp[-|x|^{1+\varepsilon}])$ implies $u\equiv 0$ (\cite[Question 1]{K05}, \cite[p. 28]{K07}, see Question \ref{q:Kenig}) when $q\in L^{\infty}(\R^n)$ is \textit{real-valued}.  

In $\R^2$, for a bounded potential $q\ge 0$ (which implies that $-\Delta+q\ge 0$), Kenig, Luis Silvestre, and Jenn-Nan Wang \cite{KSW15} improved the quantitative estimate \eqref{eq:Carleman} to $\exp[-cR\log R]$ in the right hand side of the inequality, and hence proved Landis' conjecture with a slower decay rate than that of Bourgain--Kenig.  They considered a more general operator than $H$ by allowing a drift term, and they confirmed Landis' conjecture for any real solution $u$ to 
\begin{equation}
\label{eq:drift}
\Delta u -\nabla( W u)-qu=0\quad \text{ in } \R^2,
\end{equation}
where $W(w,y)=(W_1(x,y), W_2(x,y))$ and $q(x,y)$ are measurable, real-valued, and $q(x,y)\ge 0$ a.e.. Their main result is the following theorem. 

\begin{theorem}{\cite[Theorem 1.2]{KSW15}}
\label{thm:KSWglobal}
   Assume that $(W_1(x,y), W_2(x,y))$ and $q(x,y)$ are measurable, real-valued, and $q(x,y)\ge 0$ a.e. in $\R^2$, and that
    $
    \|W\|_{L^{\infty}(\R^2)}\le 1, \|q\|_{L^{\infty}(\R^2)}\le 1$.
    Let $u$ be a real solution to 
   \eqref{eq:drift}.
    Assume that $ |u|\le \exp[C_0|z|]$ and 
   $u(0)=1$, where $z=(x,y)$. Let $z_0=(x_0,y_0)$. 
    Then we have that
    $$
     \inf_{|z_0|=R}\sup_{|z-z_0|<1}|u(z)|\ge \exp[-CR\log R]\quad \text{ for } R\gg 1, \quad \text{ where }  C \text{ depends on } C_0.
    $$
\end{theorem}

Let $B_r(a)$ denote the ball in $\R^2$ centered in $a$ and radius $r$; in the case when $a=0$, we simply denote $B_r(0)=B_r$.
The proof of Theorem \ref{thm:KSWglobal} is based on proving quantitative lower bounds in a small ball $B_r$ for solutions to a local problem posed in $B_2$, and then, the quantitative form of Landis' conjecture follows by the same scaling argument as in \cite{BK05}, see Subsection \ref{sub:BKquant}. We present the following local quantitative result.

\begin{thm}{\cite[Theorem 1.1]{KSW15}}
\label{thm:KSWlocal}
    Assume that $(W_1(x,y), W_2(x,y))$ and $q(x,y)$ are measurable, real-valued, and $q(x,y)\ge 0$ a.e. in $B_2$ and that there exist $K\ge1$, $M\ge1$ such that
    $\|W\|_{L^{\infty}(B_2)}\le K, \|q\|_{L^{\infty}(B_2)}\le M$.
    Let $u$ be a real solution to 
    \begin{equation}
    \label{eq:localdrift}
    \Delta u-\nabla (Wu)-qu=0\quad \text{ in } B_2.
\end{equation}
    Assume that $ \|u\|_{L^{\infty}(B_2)}\le K\exp[C_0(\sqrt{M}+K)]$ and 
    $
     \|u\|_{L^{\infty}(B_1)}\ge 1$.
    Then 
    \begin{equation}
    \label{eq:optimaldrift}
     \|u\|_{L^{\infty}(B_r)}\ge r^{C(\sqrt{M}+K)}
    \end{equation}
    for all sufficiently small $r$, where $C$ depends on $C_0$.
\end{thm}
We briefly comment on the exponents and related literature. The exponent $\sqrt{M}+K$ in \eqref{eq:optimaldrift} is known to be optimal. For the case where $u$ is a $\lambda$-eigenfunction of the Laplace--Beltrami operator in a smooth compact Riemannian manifold without boundary, the maximal vanishing order of $u$ is less than $C\sqrt{\lambda}$, as proved by Donnelly and Fefferman \cite{DF88}, based on Carleman's method. On the other hand, using the method of the frequency function developed by Garofalo and Lin \cite{GL86, GL87}, Kukavica \cite{K98} proved that the maximal vanishing order of $u$ solving $Lu+qu=0$ in $\Omega \subset \R^n$, $n\ge 2$, is less than $C(1+\|q_{-}\|_{L^{\infty}(\Omega)}^{1/2}+(\operatorname{osc}_{\Omega}q)^2)$, where $L$ is a general uniform second order elliptic operator, $q_{-}=\max\{-q,0\}$, and $\operatorname{osc}_{\Omega}q=\sup_{\Omega}q-\inf_{\Omega}q$; see also \cite{Z16}, where the author improves the vanishing order of solutions for Schr\"odinger equations which describes the quantitative behavior of the strong unique continuation property.

The main idea in \cite{KSW15} for the proof of Theorem \ref{thm:KSWlocal}  lies in the good relation between second order elliptic equations in the plane and the Beltrami system. The assumption $q\ge 0$ allows the authors to construct a global positive multiplier which enables them to convert \eqref{eq:localdrift} into an elliptic equation in divergence form. Through this relation with Beltrami's system, Kenig et al. can derive three-ball inequalities with optimal exponent, constant, and radii, for solutions $u$ of \eqref{eq:localdrift} without using Carleman's estimate or the frequency function. Utilizing the core idea in \cite{KSW15} with appropriate modifications, a series of results on the weak Landis' conjecture in $\R^2$ have been obtained for more general second-order elliptic operators with  bounded potential $q$ which is allowed to be sign-changing, for instance in \cite{Davey20, DKW17, DKW20, DW20,KW15}.

\subsection{Ideas and sketch of the proofs of Theorems \ref{thm:KSWlocal} and \ref{thm:KSWglobal}}

To derive the improved quantitative estimate in Theorem \ref{thm:KSWlocal}, instead of a Carleman-type inequality, Kenig et al. exploited the nonnegativity of the operator in \eqref{eq:H} with $q\ge0$, and in particular, they used a \textit{three-ball inequality} derived from \textit{Hadamard's three-circle theorem}. The strategy is motivated by a good understanding of the problem in the context of Beltrami's system. We define $\bar{\partial}=
\frac12(\partial_x+i\partial_y)$. A first enlightening result is the following. 

\begin{thm}{\cite[Theorem 2.1]{KSW15}}
Let $u$ be any solution to 
$$
\bar{\partial}u=qu\quad \text{ in } \quad \R^2.
$$
Assume that $\|q\|_{L^{\infty}(\R^2)}\le 1$, $|u(z)|\le \exp[C_0|z|]$, and $u(0)=1$. Then
\begin{equation}
\label{eq:quantiBeltrami}
  \inf_{|z_0|=R}\sup_{|z-z_0|<1}|u(z)|\ge \exp[-CR\log R]\quad \text{ for } R\gg 1, \quad \text{ where }  C \text{ depends on } C_0.
\end{equation}
\end{thm}

\begin{proof}
    We begin with the local equation
   \begin{equation}
   \label{eq:beltramilocal}
\bar{\partial}u=qu\quad \text{ in } \quad B_2\quad \text{ with }\quad \|q\|_{L^{\infty}(B_2)}\le M.
\end{equation}
Any solution to \eqref{eq:beltramilocal} can be written as 
$
u=\exp[w]f
$,
with $f$ holomorphic in $B_2$ and 
$$
w(z)=-\frac{1}{\pi}\int_{B_2}\frac{q(\zeta)}{\zeta-z}\,d\zeta.
$$
By Hadamard's three-circle theorem (see, for example, \cite{A78}), we have
$$
\|f\|_{L^{\infty}(B_{r_1})}\le \|f\|_{L^{\infty}(B_{r})}^{\theta}\|f\|_{L^{\infty}(B_{r_2})}^{1-\theta},
$$
where $r<r_1<r_2$ and 
$
\theta=\frac{\log(\frac{r_2}{r_1})}{\log(\frac{r_2}{r})}
$.
The choice $r_1=1$ and $r_2=\frac32$ yields
$
\exp[-CM]\le \|u\|_{L^{\infty}(B_{r})}^{\theta}
$
and hence
\begin{equation}
\label{eq:quantilocalBel}
\|u\|_{L^{\infty}(B_{r})}\ge r^{CM}, \quad \text{ where }  C \text{ depends on } C_0.
\end{equation}

Now we use the scaling argument in \cite{BK05} (see Subsection \ref{sub:BKquant}), which we briefly recall here. Denote $|z_0|=R$. Let $u_R(z)=u\big(R(z+\frac{z_0}{R})\big)$, then $u_R$ satisfies the equation \eqref{eq:beltramilocal} with a rescaled potential $q_R(z)=Rq\big(R(z+\frac{z_0}{R})\big)$, where 
$\|q_R\|_{L^{\infty}(B_2)}\le R$.
It can be seen that $|u_R(z)|\le \exp[CC_0R]$ for all $z\in B_2$ and 
$
u_R\big(-\frac{z_0}{R}\big)=u(0)=1$
with $\big|\frac{z_0}{R}\big|=1$. Hence, we have 
$$
\|u_R\|_{L^{\infty}(B_1)}\ge 1
$$
and the quantitative estimate \eqref{eq:quantiBeltrami} follows from \eqref{eq:quantilocalBel} by taking $r=R^{-1}$, namely
\begin{equation*}
\|u_R\|_{L^{\infty}(B_{R^{-1}})}\ge R^{-CR}.
\end{equation*}
\end{proof}

We turn to Theorems \ref{thm:KSWlocal} and \ref{thm:KSWglobal}. We will first show a sketch of the proofs without the presence of $W$ based on the ideas for Beltrami's system.
\begin{thm}
\label{thm:toywithout}
    Let $u$ be a real solution to 
\begin{equation}
\label{eq:otravezlocal}
\Delta u-q(z)u=0\quad \text{ in } \quad B_2\subset \R^2,
\end{equation}
with $q(z)\ge0$ and $\|q\|_{L^{\infty}(B_2)}\le M$, for some $M\ge1$.  Assume that $u$ satisfies
\begin{equation}
    \label{eq:cotaB2}
    \|u\|_{L^{\infty}(B_{2})}\le \exp[C_0\sqrt{M}]
\end{equation}
for some $C_0>0$ and furthermore $\|u\|_{L^{\infty}(B_1)}\ge1$. Then we have that
$$
\|u\|_{L^{\infty}(B_r)}\ge r^{C\sqrt{M}}, \quad \text{ where }  C \text{ depends on } C_0.
$$
\end{thm}
\begin{proof}
The first crucial step is that thanks to the non-negativity of $q$, it can be proven that there exists a positive solution $\phi$ satisfying \eqref{eq:otravezlocal} and 
\begin{equation}
\label{eq:estimacionphi}
\exp[-2\sqrt{M}]\le \phi(z)\le \exp[2\sqrt{M}], \quad \text{ for all } z\in B_2.
\end{equation}
Set now $u=\phi v$, then $v$ satisfies
$
\nabla\cdot (\phi^2\nabla v)=0$ in $B_2$.
Let $\widetilde{v}$ with $\widetilde{v}(0)=0$ be such that
$$
\begin{cases}
\partial_y \widetilde{v}=\phi^2\partial_xv,\\
-\partial_x \widetilde{v}=\phi^2\partial_yv.
\end{cases}
$$
Let $g=\phi^2v+i\widetilde{v}$, then $g$ satisfies
\begin{equation}
    \label{eq:eqcompleja}
    \bar{\partial}g= \bar{\partial}\phi^2v=\frac{ \bar{\partial}\phi^2}{2\phi^2}(g+\bar{g})\quad \text{ in } B_2.
\end{equation}
On the other hand, let $\alpha=\frac{ \bar{\partial}\phi^2}{2\phi^2}=\frac{ \bar{\partial}\phi}{\phi}=\bar{\partial}\log (\phi)$,
thus \eqref{eq:eqcompleja} is equivalent to 
\begin{equation}
    \label{eq:equivalente}
    \bar{\partial}g=\alpha g+\alpha \bar{g}\quad \text{ in } B_2.
\end{equation}
Defining 
$$
\widetilde{\alpha}=\begin{cases}
\alpha+\alpha\bar{g}/g, \quad \text{ if } g\neq 0,\\
0, \quad \text{ otherwise},
\end{cases}
$$
then \eqref{eq:equivalente} is reduced to 
\begin{equation}
\label{eq:finalcomp}
\bar{\partial}g=\widetilde{\alpha}g\quad \text{ in } B_2.
\end{equation}

We need a precise estimate for $\alpha=\bar{\partial} \log(\phi)$. If we denote $\psi=\log \phi$ we see, in view of \eqref{eq:estimacionphi}, that $\psi$ satisfies 
$
|\psi(z)|\le 2\sqrt{M}$ in $ B_2$
and solves the equation $\Delta\psi+|\nabla \psi|^2=q$ in $B_2$. The crucial estimate $\|\nabla \psi\|_{L^{\infty}(B_{7/5})}\le C\sqrt{M}$, where $C>0$ is an absolute constant, is proven in \cite[Lemma 2.2]{KSW15}. Using this, we have $\|\alpha\|_{L^{\infty}(B_{7/5})}\le C\sqrt{M}$, which implies $\|\widetilde{\alpha}\|_{L^{\infty}(B_{7/5})}\le C\sqrt{M}$. Let $w(z)$, $z=x+iy$, be defined by 
$$
w(z)=-\frac{1}{\pi}\int_{B_{7/5}}\frac{\widetilde{\alpha}(\xi)}{\xi-z}\,d\xi.
$$
Then $\bar{\partial}w=\widetilde{\alpha}$ in $B_{7/5}$ and $\|w\|_{L^{\infty}(B_{7/5})}\le C\|\widetilde{\alpha}\|_{L^{\infty}(B_{7/5})}\le C\sqrt{M}$. Any solution $g$ of \eqref{eq:finalcomp} in $B_{7/5}$ is given by $g(z)=\exp[w(z)]h(z)$, where $h$ is holomorphic in $B_{7/5}$. Applying Hadamard's three-circle theorem to $h=\exp[-w]g$ and the estimates above we obtain 
$$
\|g\|_{L^{\infty}(B_{r_1})}\le\exp[C\sqrt{M}] \|g\|_{L^{\infty}(B_{r/2})}^{\theta}\|f\|_{L^{\infty}(B_{r_2})}^{1-\theta},
$$
where $\theta=\frac{\log(\frac{r_2}{r_1})}{\log(\frac{2r_2}{r})}$.
With the appropriate choice of $r, r_1$, and $r_2$, recalling the definition of $g$ and using interior estimates for $\phi$ and $u$, it is possible to conclude that 
$$
\|u\|_{L^{\infty}(B_{1})}\le \exp[C\sqrt{M}]\|u\|_{L^{\infty}(B_{r})}^{\theta},
$$
where $C$ depends on $C_0$. From here, the thesis follows.
\end{proof}

Using the scaling argument of \cite{BK05} as above, Theorem \ref{thm:toywithout} implies a quantitative version of Landis' conjecture as in Theorem \ref{thm:KSWglobal} with $W=\bar{0}$, see \cite[Theorem 2.6]{KSW15}. For the proof of Theorem \ref{thm:KSWglobal} we can follow the scheme above when considering equation \eqref{eq:drift} with the assumptions on Theorem \ref{thm:KSWlocal}. To construct a positive multiplier for \eqref{eq:drift}, we consider the adjoint operator, i.e.
$$
L^*u=\Delta u+W\cdot \nabla u-qu.
$$
As above, letting $u=\phi v$, then $v$ satisfies
$
\nabla\cdot \big(\phi^2(\nabla v-Wv)\big)=0$, in  $B_2$.
Let $\widetilde{v}$ with $\widetilde{v}(0)=0$ be such that
$$
\begin{cases}
\partial_y \widetilde{v}=\phi^2\partial_xv-\phi^2W_1v,\\
-\partial_x \widetilde{v}=\phi^2\partial_yv-\phi^2W_2 v.
\end{cases}
$$
Let $g=\phi^2v+i\widetilde{v}$, then $g$ solves
\begin{equation}
    \label{eq:eqcomplejabis}
    \bar{\partial}g= \gamma( g+\bar{g})\quad \text{ in } B_2,
\end{equation}
where $\gamma=\bar{\partial}\log(\phi)+\frac14(W_1+iW_2)$.
Defining 
$$
\widetilde{\gamma}=\begin{cases}
\gamma+\gamma\bar{g}/g, \quad \text{ if } g\neq 0,\\
0, \quad \text{ otherwise},
\end{cases}
$$
then \eqref{eq:eqcomplejabis} becomes 
\begin{equation}
\label{eq:finalcompbis}
\bar{\partial}g=\widetilde{\gamma}g\quad \text{ in } B_2.
\end{equation}
Let $\phi=\exp[\psi]$, then $\psi$ satisfies 
$
|\psi(z)|\le 2(\sqrt{M}+K)$ in $B_2$ 
and solves the equation $\Delta\psi+|\nabla \psi|^2+W\cdot \nabla \psi=q$ in $B_2$. Similarly, the estimate $\|\nabla \psi\|_{L^{\infty}(B_{7/5})}\le C(\sqrt{M}+K)$, where $C>0$ is an absolute constant, is proven in \cite[Lemma 3.1]{KSW15}. Using this, we have that  $\|\widetilde{\gamma}\|_{L^{\infty}(B_{7/5})}\le C(\sqrt{M}+K)$. Therefore, any solution $g$ of \eqref{eq:finalcompbis} in $B_{7/5}$ can be represented as $g(z)=\exp[\widetilde{w}(z)]h(z)$, where $h$ is holomorphic in $B_{7/5}$ and $\|\widetilde{w}\|_{L^{\infty}(B_{7/5})}\le C(\sqrt{M}+K)$. The remainder of the proof follows as in the proof of Theorem \ref{thm:toywithout}.

In addition to the case of the whole plane, Kenig, Silvestre, and Wang also establish a quantitative version of Landis' conjecture in an exterior domain \cite[Theorem 1.5]{KSW15}. Here, the scaling argument fails to imply the conjecture, and the main tool is instead a Carleman estimate. We point out that a nice summary of \cite{KSW15} can be read in \cite{MAGF18}.
The techniques from complex analysis used in \cite{KSW15} have been used in further works concerning observability inequalities closely related to Landis' conjecture, see \cite{ZZ24}.

\section{Weak Landis' conjecture is true in dimension two}
\label{sec:logunovetal}

The approach in \cite{KSW15} relied on the assumption $q\ge 0$. In \cite[Section 6]{KSW15}, Kenig et al. suggested an alternative to remove such an assumption, which was based on considering a new function $u_{\xi}(x,y)=\cosh(\lambda \xi)u(x,y)$, $\lambda\ge \sqrt{\|q\|_{L^{\infty}}}$, such that if $u$ satisfies $\Delta_{x,y}u+q u=0$, then $u_{\xi}$ satisfies a new equation with the non-negative potential $\lambda^2-q$. Consequently, a new global positive solution can be constructed, but an optimal three-balls inequality is not available. New ideas appeared to be needed. 

In 2020, Alexander Logunov, Malinnikova, Nicolai Nadirashvili, and Fedor Nazarov \cite{LMNN20} managed to remove the assumption $q\ge 0$ and confirmed the weak version of Landis' conjecture in dimension two. 
\begin{theorem}[{\cite[Theorem 1.1]{LMNN20}}]
\label{thm:LMNN}
      Suppose that $-\Delta u+q u=0$ on $\R^2$, where $u$ and $q$ are real-valued and $|q|\le 1$. If $|u(x)|\le  C\exp[-|x|\log^{1/2}|x|]$, $|x|>2$, where $C$ is a sufficiently large absolute constant, then $u\equiv0$.  
\end{theorem}

\begin{rmk}
   In a private communication after completing this manuscript, F. Nazarov sent us an example \cite{LN24} of a real non-zero function satisfying $|\Delta u|\le |u|$ and $|u(x)|\le C \exp[-c|x| \log^{1/2}|x|)]$, therefore showing the optimality of Theorem \ref{thm:LMNN} and hence disproving the (strong) Landis' conjecture in the plane.
    \end{rmk}

The proof of Theorem \ref{thm:LMNN} combines techniques of quasiconformal mappings and the structure of the nodal set (zero set) of the solution $u$. In particular, the method involves nodal sets of $u$ and holes that are made in nodal domains (connected components of the complement of the zero set), so the problem is reduced to a simpler problem in certain perforated domains. Some two-dimensional tools are used in the proof and weak Landis' conjecture in higher dimensions is still open in its full generality. Logunov et al. also proved a quantitative local version of Landis' conjecture which improves the one by Bourgain and Kenig in Theorem~\ref{thm:BKquantitative}.

\begin{thm}[{\cite[Theorem 1.2]{LMNN20}}]
\label{thm:quantLMNN}
Let $u$ be a real solution to \eqref{eq:H} in $B(0,2R)\subset \R^2$, where $q$ is real-valued and $|q|\le 1$. Suppose that $|u(0)|=\sup_{B(0,2R)}|u|=1$. Then for any $x_0$ with $|x_0|=R/2>2$, we have
$$
\sup_{B(x_0,1)}|u|\ge \exp[-CR\log^{3/2}R]
$$
with some absolute constant $C>0$.
\end{thm}
Theorem \ref{thm:LMNN} and Theorem \ref{thm:quantLMNN} follow from  Theorem \ref{thm:local}, where one does not assume that $|u(0)|=\sup_{B(0,2R)}|u|=1$ and it can be understood as a version of a three-balls inequality.

\begin{thm}[{\cite[Theorem 2.2]{LMNN20}}]
\label{thm:local}
Let $u$ be a solution to \eqref{eq:H} in $B(0,R)\subset \R^2$, $R>2$, $q$ is real-valued, $|q|\le 1$, and 
\begin{equation}
\label{eq:assumptionLMNN}
\frac{\sup_{B(0,R)}|u|}{\sup_{B(0,R/2)}|u|}\le e^N,
\end{equation}
then
\begin{equation}
\label{eq:keyLMNN}
\sup_{B(0,r)}|u|\ge (r/R)^{C(R\log^{1/2}R+N)}\sup_{B(0,R)}|u|
\end{equation}
for any $r<R/4$, with some absolute constant $C>0$.
\end{thm}

In turn, the proof of Theorem \ref{thm:local} consists of three parts (or \textit{Acts}, as named by the authors). 

\subsection{Ideas in the proof of Theorem \ref{thm:local}}

\noindent \textit{Act I. Construction of a positive multiplier in a suitable perforated domain.} This first step consists of constructing a positive solution (multiplier) of the equation $\Delta \varphi +q\varphi=0$ (alike in \cite{KSW15}) in a certain domain. Such a construction is possible by perforating the domain $B(0,R)$ in a suitable way, so that the Poincar\'e constant is small and explicitly quantified. Then, in the perforated domain, due to this precise control of the (small) Poincar\'e constant, one can construct such a positive solution. 

Let us provide details. Consider the nodal set 
$
F_0=\{x:u(x)=0\}
$
and assume that $u$ is a solution to \eqref{eq:H} in $B(0,R)$, $R>1$. Take $\varepsilon>0$, which is a parameter chosen appropriately at the end of the proof, and add finitely many $C\varepsilon$-separated closed disks of radius $\varepsilon$ to $F_0$ so that the distance from each disk to $F_0$ is greater or equal than $C\varepsilon$. Denote by $F_1$ the union of the closed disks. Then
$
F_0\cup F_1 \cup \{z:|z|\ge R\}
$
is a $3C\varepsilon$-net on the plane, assuming $C>2$.

\begin{figure}
\includegraphics[scale=0.25]{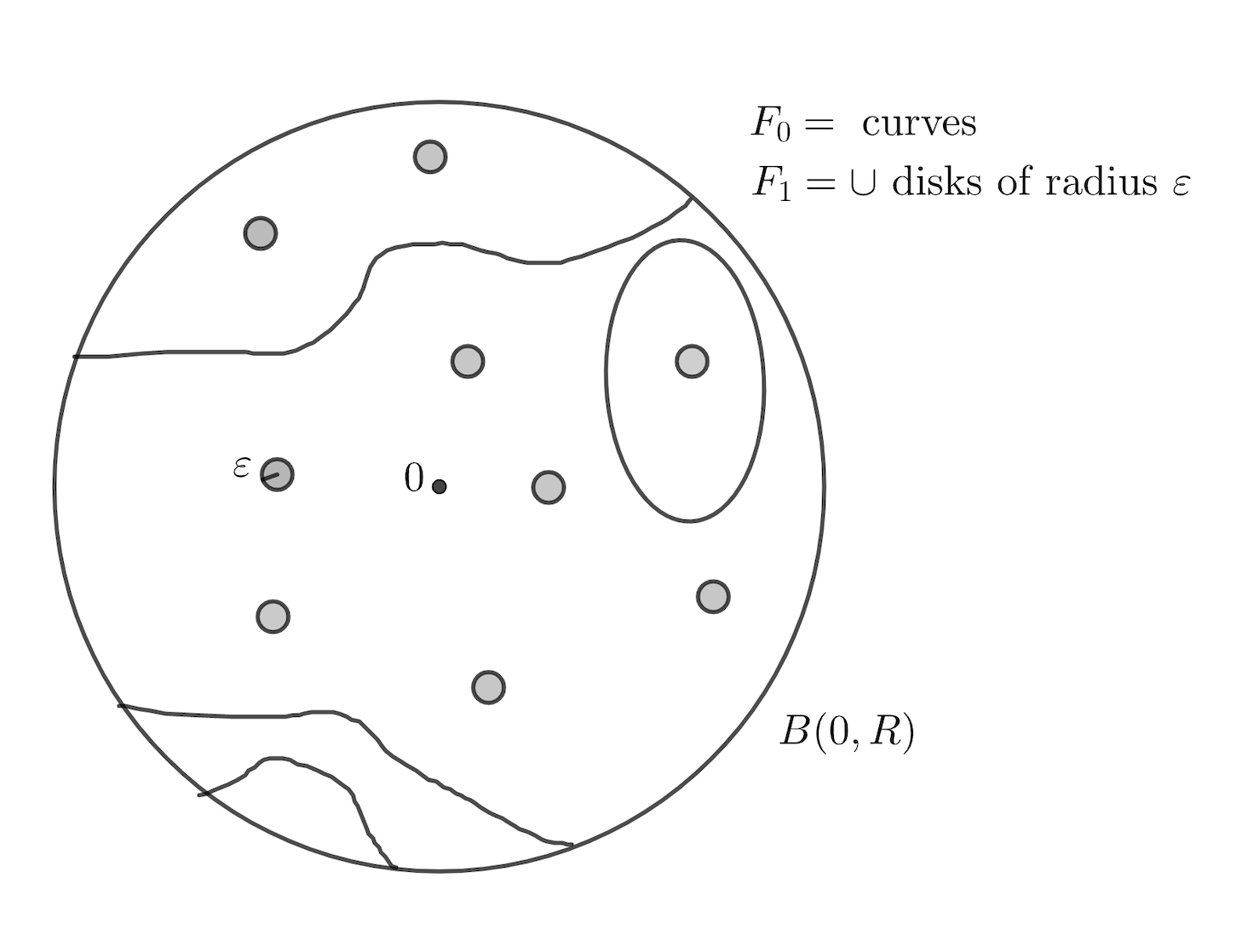}
\caption{The punctured domain $\Omega$}
  \label{fig:punctured}
  \end{figure}
Moreover, it can be shown that the domain $\Omega$, which consists of perforating the domain $B(0,R)$ with the small disks $F_1$ and $F_0$ (see Figure \ref{fig:punctured})
$$
\Omega=\{z:|z|<R, z\notin F_0\cup F_1\}
$$
is an open, possibly disconnected set with Poincar\'e constant controlled by $C'\varepsilon^2$, where $C'$ only depends on $C$, namely for every $g\in W^{1,2}_0(\Omega)$, we have\footnote{Observe that the function $g$ in the inequality is not the solution $u$ to $\Delta u+q u=0$ in $B(0,R)$.}
$$
\int_{\Omega}g^2\le C'\varepsilon^2\int_{\Omega}|\nabla g|^2.
$$
The properties of the nodal set and the perforation process allow one to construct a function (multiplier) $\varphi$ in $B(0,R)$ such that 
$\varphi\in W^{1,2}(\Omega)$ and
$$
\Delta \varphi+q \varphi=0 \text{ in } \Omega,
\qquad \text{ and }\quad \widetilde{\varphi}=\varphi-1\in W^{1,2}_0(\Omega) \quad \text{ satisfies }
\quad \|\widetilde{\varphi}\|_{\infty}\le C''\varepsilon^2,
$$
so that $\varphi$ is a positive function (because $\varepsilon$ is small).

\noindent \textit{Act II. Reduction to a non-homogeneous Beltrami equation.} The second step consists of a reduction of the problem to a divergence type equation in a domain with holes and a quasiconformal mapping with a control of distorsion $K$ in terms of the radius of the ball $B(0,R)$. Thus, first it is proven that, thanks to the positive multiplier of the previous step, $f=\frac{u}{\varphi}$ is a solution to 
$$
\operatorname{div}(\varphi^2\nabla f)=0\qquad \text{ in } \Omega
$$
and $f\in W^{1,2}_{\operatorname{loc}}(B(0,R))$.
It turns out that the equation above holds even through $F_0$ in the whole $B(0,R)\setminus F_1$. The divergence elliptic equation is then transformed into a Beltrami equation, thanks to Stoilow's factorization theorem, through a $K$-quasiconformal change of variables $g$ mapping $0$ to $0$ and $B(0,R)$ into $B(0,R)$ and such that 
$$
f=h\circ g
$$
with $h$ a harmonic function in $B(0,R)\setminus g(F_1)$, and $K$ is close to $1$ when $\|\varphi-1\|_{\infty}$ is small, that is
$$
K\le \frac{1+\Big\|\frac{1-\varphi^2}{1+\varphi^2}\Big\|_{\infty}}{1-\Big\|\frac{1-\varphi^2}{1+\varphi^2}\Big\|_{\infty}}\le 1+C\varepsilon^2.
$$
Then one chooses 
$
\varepsilon\sim \frac{1}{\sqrt{\log R}},
$
so that the distortion on scales from $1/R$ to $R$ is bounded and the images of the disks in $F_1$ have size comparable to $\varepsilon$. We emphasize that a harmonic function $h$ in $B(0,R)\setminus g(F_1)$ is obtained, where $g(F_1)$ is the union of sets of diameter  $\varepsilon$ and each set (the image of a single disk) is surrounded by an annulus of width $C\varepsilon$ in which $h$ \textit{does not change sign}.

 \noindent \textit{Act III. A local reformulation with two dimensional harmonic functions.} After rescaling, we end up with the function $h$, which is harmonic in a punctured domain $B(0,R')\setminus \cup_jD_j$, with $R'\sim \frac{R}{\varepsilon}\sim R\sqrt{\log R}$ and $D_j$ are $1000$-separated unit disks. If we assume also that $h$ does not change sign in $5D_j\setminus D_j$, with the assumption \eqref{eq:assumptionLMNN}, it can be concluded that
 $$
\frac{\sup_{B(0,R')\setminus{\cup_j3D_j}}|h|}{\sup_{\{z: R'/8<|z|<R'\}\setminus{\cup_j3D_j}}|h|}\le \exp[CR'].
 $$
 The goal of Theorem \ref{thm:local} is to estimate $\sup_{B(0,R)}|u|$ from below, and in order to prove the estimate \eqref{eq:keyLMNN}, it is enough to show that 
 $$
 \sup_{B(0,r')\setminus \cup_j3D_j}|h|\ge c(r'/R')^{C(R'+N')}\sup_{B(0,R')\setminus \cup_j3D_j}|h|,
 $$
 where $r'\ge \frac{R'}{16}\big(\frac{r}{R}\big)^2$. Therefore, the proofs of the main theorems have been reduced to prove a local, harmonic counterpart version of the main theorem, given in \cite[Theorem 5.3]{LMNN20}, which is the core of the matter, and where the separation of disks and preservation of sign are crucial. 

  \noindent \textit{Final Act.} We see how Theorem \ref{thm:local} implies Theorems \ref{thm:LMNN} and \ref{thm:quantLMNN}. For Theorem \ref{thm:LMNN}, we can assume that $|u|$ attains its global maximum at some point on the plane, otherwise $|u|$ does not tend to $0$ near infinity. Denoting $|u(z_{\max})|=\max_{\R^2}|u|=1$, for any $R>6|z_{\max}|$ and any $x$ with $|x|=R/3$, we have 
  \begin{equation}
  \label{eq:maxim}
  \sup_{B(x,R)}|u|= \sup_{B(x,R/2)}|u|=1.
  \end{equation}
If additionally $R>2$ then by Theorem \ref{thm:local} applied to $u(\cdot + x)$, we have $ \sup_{B(x,R/4)}|u|\ge e^{-CR\log^{1/2}R}$, thus
  $$
  \sup_{|z|>R/12}|u|\ge \exp[-CR\log^{1/2}R].
  $$
  In order to deduce Theorem \ref{thm:quantLMNN} we again note that \eqref{eq:maxim} holds for any $x$ with $|x|=R/2$, because $|u(0)|=\max_{B(0,2R)}|u|$. Applying Theorem \ref{thm:local} to $u(\cdot+x)$, we get 
  $$
  \sup_{B(x,1)}|u|\ge R^{-CR\log^{1/2}R} =\exp[-CR\log^{3/2}R].
  $$
  which is the thesis in Theorem \ref{thm:quantLMNN}.

A nice summary of the strategy of \cite{LMNN20} can be also found in \cite[Chapter 24]{Hsu22}. Relying heavily on the ideas presented in \cite{LMNN20}, Davey \cite{Davey23} extended Logunov's et al. result, also for $n = 2$ and real potentials, by allowing the potentials to grow at infinity. More precisely, for equations of the form $-\Delta u+qu=0$ in $\R^2$, with $|q(x)|\lesssim |x|^N$ for some $N\ge0$, Davey proves that real-valued solutions satisfy exponential decay estimates with a rate that depends explicitly on $N$. The case $N=0$ corresponds to Landis' conjecture for real-valued solutions in the plane \cite{LMNN20} described in this section. 

Also inspired by the methodology of proof of \cite{LMNN20}, Le Balc'h and Souza \cite{BS24} recently studied a quantitative form of the Landis conjecture on exponential decay for real-valued solutions to second order elliptic equations with variable coefficients in the plane. They follow the steps in \cite{LMNN20} in a rather neat way. In particular, they prove a qualitative form of Landis' conjecture  for real-valued weak solutions to $-\Delta u-\nabla\cdot (W_1 u)+W_2\cdot \nabla u+q u=0$ in $\R^2$, where $W_i\in L^{\infty}(\R^2;\R^2)$, $i=1,2$ are drifts and $q\in L^{\infty}(\R^2;\R)$: if $|u(x)|\le \exp[-|x|^{1+\delta}]$ for $\delta>0$, $x\in \R^2$, then $u\equiv 0$. The presence of drifts forces twists in Le Balc'h and Souza's strategy comparing to Logunov et al.'s strategy, and their last step relies on a Carleman inequality for the Laplacian on perforated domains rather than working with a simple harmonic equation, which is not possible to reach in this situation. A version with potentials $q\in L^p(\R^2;\R)$ and $p\in (,+\infty]$ is proven in \cite{LB24}.

The techniques developed by Logunov et al. seem to be difficult to extend in higher dimensions to get the sharp results towards  Landis' conjecture. Nevertheless, it is interesting to note that they also present a result for harmonic functions in higher dimensions, based on a Carleman estimate, with a slightly worse bound than the one in the complex plane \cite[Theorem 5.1 (Toy problem)]{LMNN20}. 

\begin{thm}{\cite[Toy problem in higher dimensions with extra logarithm]{LMNN20}}
\label{thm:worseLMNN}
    Let $\{D_j\}$ be a collection of $100$-separated balls on $\R^n$ and let $R>100$. Let $u$ be a harmonic function in $B_R\cup 5D_j\setminus D_j$ such that $u$ does not change sign in each $B_R\cap 5D_j\setminus D_j$.
    If $|u(z)|\le \exp[-L|z|\log|z|]$ for all $z\in B_R\cup 5D_j\setminus D_j$ and $L$ is sufficiently large, then $u\equiv 0$.
\end{thm}

The proof of Theorem \ref{thm:worseLMNN} is based on a Carleman inequality with log linear weight, which is an exceptionally relaxed assumption, since most of Carleman's inequalities require strict log convexity properties. Methodology in \cite{LMNN20} has inspired observability estimates for non-homogeneous elliptic equations in the presence of a potential, posed on a smooth bounded domain $\Omega$ in $\R^2$ and observed from a non-empty open subset $\omega\subset \Omega$, see \cite{EB23}. The main difference between \cite{EB23} and \cite{LMNN20} is that the zero set of solutions to elliptic equations with source terms can be very intricate and must be handled more carefully. Technical improvement of the Carleman's estimate (actually, additional treatment via Harnack inequalities) in \cite{EB23} enables to generalize the toy problem in Theorem \ref{thm:worseLMNN} without logarithm loss.

\section{Anderson--Bernoulli model and Landis' conjecture on graphs}
\label{sec:AB_discreto}

\subsection{Anderson--Bernoulli model in the two and three dimensional lattices}
\label{subsec:Ding-Smart}

It was mentioned in Section \ref{BK} that the case of the lattice version of Anderson--Bernoulli model (i.e., the original model considered by Anderson with Bernoulli potential) for $n\ge 2$ remained unsettled at the time of writing \cite{BK05}. About 15 years later, J. Ding and C. K. Smart, in their breakthrough work in \cite{DS20}, considered the 
random Schr\"odinger operator on $\ell^2(\Z^2)$ given in \eqref{eq:Hd} with $q$ being a random Bernoulli potential, and which we will denote in this subsection by $
H_{\operatorname{2D\ell}}$.

The spectrum of the discrete Laplacian is the closed interval $\sigma(-\Delta_{\operatorname{d}})=[0,4n]$ and the spectrum of the 
random Hamiltonian is almost surely the closed interval $\sigma(H_{\operatorname{2D\ell}})=[0,4n+\lambda]$.
While the spectrum of the discrete Laplacian is absolutely continuous, the perturbation $\lambda q$ can create eigenvalues. Moreover, the perturbation may lead to eigenfunctions that are exponentially localized within an interval in $\sigma(H_{\operatorname{2D\ell}})$, in which case $H_{\operatorname{2D\ell}}$ is said to have Anderson localization.  In short, we say that $H_{\operatorname{2D\ell}}$ has ``Anderson localization'' in the spectral interval $I\subseteq\sigma(H_{\operatorname{2D\ell}})$ if $
\psi: \Z^n\to \R, \lambda\in I,  H_{\operatorname{2D\ell}}\psi=\lambda\psi$, and 
$$
\inf_{p>0}\sup_{x\in \Z^n}(1+|x|)^{-p}|\psi(x)|<\infty\quad
\text{ implies }\quad 
\inf_{p>0}\sup_{x\in \Z^n}e^{p|x|}|\psi(x)|<\infty.
$$
Ding and Smart proved that on a two-dimensional lattice this is indeed the case.

\begin{theorem}{\cite[Theorem 1.1]{DS20}}
\label{thm:DS}
In dimension $n=2$ there is an $\varepsilon>0$, depending on $\lambda>0$, such that, almost surely, $H_{\operatorname{2D\ell}}$ has Anderson localization in $[0,\varepsilon]$.
\end{theorem}

The proof of this result follows the approach of Bourgain and Kenig \cite{BK05} which we superficially described in Section \ref{BK}, with some significant modifications. Ding and Smart performed a multiscale analysis, keeping track of a list of ``frozen'' sites $F\subseteq \Z^2$ where the potential has been already sampled. The complementary ``free'' sites $\Z^2\setminus F$ are sampled only to perform an eigenfunction variation on rare, ``bad'' squares. This strategy of frozen and free sites is used to obtain a version of the Wegner \cite{W81} estimate that is otherwise unavailable in the Bernoulli setting. 

We recall that the eigenvalue variation of Bourgain--Kenig \cite{BK05} relies crucially on an a priori quantitative unique continuation result, Theorem \ref{thm:BKquantitative}. We also recall that the corresponding fact is false on the lattice, so one needs a substitute for the quantitative unique continuation result for eigenfunctions of $H_{\operatorname{2D\ell}}$. For the two-dimensional lattice $\Z^2$, a hint of the missing ingredient appears in the work by L. Buhovsky, Logunov, Malinnikova, and M. A. Sodin \cite{BLMS22}. In \cite{BLMS22}, it is proved that any function $u:\Z^2\to \R$ that is harmonic and bounded on a $1-\varepsilon$ fraction of sites must be constant. One of the key components of this Liouville theorem is a quantitative unique continuation result for harmonic functions on the two dimensional lattice, which can be formulated as follows. 

\begin{thm}{\cite{BLMS22}}
\label{thm:BLMS22}
    There are constants $\alpha>1>\varepsilon>0$ such that, if $u:\Z^2\to \R$ is a lattice harmonic function in a square $Q\subseteq \Z^2$ of side length $L\ge \alpha$, then
    $$
    |\{x\in Q: |u(x)|\ge e^{-\alpha L}\|u\|_{\ell^{\infty}(\frac12Q)}\}|\ge \varepsilon L^2.
    $$
\end{thm}

Theorem \ref{thm:BLMS22} implies that any two lattice harmonic functions that agree on a $1-\varepsilon$ fraction of sites in a large square must be equal in the concentric half square. This result is false in dimensions three and higher. Inspired by Theorem~\ref{thm:BLMS22} and its proof, one of the main contributions of Ding and Smart is the following random quantitative unique continuation result for eigenvalues of the Hamiltonian $H_{\operatorname{2D\ell}}$.

\begin{thm}{\cite[Theorem 1.6]{DS20}}
\label{thm:DS20}
    There are constants $\alpha>1>\varepsilon>0$ such that, if $\bar{\lambda}\in [0,9]$ is an energy and  $Q\subseteq \Z^2$ is a square of side length $L\ge \alpha$, then 
    $\mathbb{P}[\mathcal{E}]\ge 1-e^{-\varepsilon L^{1/4}}$, where $\mathcal{E}$ denotes the event that
    $$
    |\{x\in Q: |\psi(x)|\ge e^{-\alpha L\log L}\|\psi\|_{\ell^{\infty}(\frac12Q)}\}|\ge \varepsilon L^{3/2}(\log L)^{-1/2}
    $$
    holds whenever $\lambda\in \R$, $\psi:\Z^2\to \R$, $|\lambda-\bar{\lambda}|\le e^{-\alpha(L\log L)^{1/2}}$, and $H_{\operatorname{2D\ell}}\psi=\lambda \psi$ in $Q$.
\end{thm}

Roughly, Theorem \ref{thm:DS20} says that, with high probability, every eigenfunction on a square $Q$ with side length $L$ is supported on at least $L^{3/2-\varepsilon}$ many points (sites) in $Q$. This weaker variant of unique continuation, which generalises Theorem \ref{thm:BLMS22}, turns out to be still enough for multiscale analysis. An important ingredient in \cite{BLMS22} is a technical lemma \cite[Lemma 3.4]{BLMS22} which gives a priori bounds on how information propagates from the boundary to the interior of a tilted rectangle, and it is used in the upper bound in Liouville's theorem for harmonic functions in the lattice. More precisely, if $u$ is harmonic in certain $45^{\circ}$-rotated rectangle, bounded on the northwest boundary, and bounded on half of the southeast boundary, then $u$ is bounded on the whole rectangle. The main idea of the proof of \cite[Lemma 3.4]{BLMS22} is that such a function $u$ can be related to a polynomial for which Remez' inequality can be applied, yielding the desired upper estimate. This strategy seems to break in the presence of a potential, and the twist in the argument by Ding and Smart is that a solution map related to $H_{\operatorname{2D\ell}}\psi=\lambda \psi$ on a very thin tilted rectangle has to be treated as a random linear operator with a large deviations estimate. With the main ingredient at hand, the next step (alike in \cite[Lemma~3.6]{BLMS22}) is a growth lemma \cite[Lemma 3.18]{DS20}, where the use of very thin rectangles leads to large support on only $L^{3/2-\varepsilon}$ many points, and which permits to run a Calder\'on--Zygmund stopping time argument \cite[Theorem 3.5]{DS20}. The latter is a random version of \cite[Proposition 3.9]{BLMS22}. 

The overall proof strategy by Ding and Smart was later used by L. Li \cite{Li22} on the 2D square lattice to prove that Anderson localization holds at large disorder outside of a finite number of intervals that shrink as the disorder strength grows: Bourgain--Kenig's multiscale argument for the 2D continuum Anderson--Bernoulli model is adapted to the discrete setting and again a critical role is played by the discrete quantitative unique continuation type result of Buhovsky, Logunov, Malinnikova and Sodin.

Recently, Li and L. Zhang \cite{LZ22} have proved a deterministic unique continuation result on three dimensions that is a sufficient substitute for Theorem \ref{thm:DS20} to prove
localization of eigenfunctions corresponding to eigenvalues near zero, the lower boundary of the
spectrum, on the 3D lattice $\Z^3$. For $r\in \R_+$, we will call the set $Q_r:=([-r,r]\cap \Z)^3$ a cube. A simplified statement of such is stated as follows. 
\begin{theorem}{\cite[Theorem 1.3]{LZ22}}
\label{thm:LZ22simp}
    There exists a constant $p>3/2$ such that the following holds. For each $K>0$, there is $C_1>0$, such that for any large enough $n\in \Z_+$ and functions $u,q:\Z^3\to \R$ with
    $\Delta_{\operatorname{d}}u=qu$ in $Q_n$ and $\|q\|_{\infty}\le K$, we have that
    $$
    |\{a\in Q_n: |u(a)|\ge \exp[-C_1 n]|u(0)|\}|\ge n^p.
    $$
\end{theorem}
The proof of this weaker variant of unique continuation for $\Z^3$ is done in several steps: 

\noindent \textit{Step 1.} In the first step, Li and Zhang use the ideas from Buhovsky et al. to prove that if $T_n$ is an equilateral, two-dimensional triangle centered at $0$, with sidelength $3n+1$, spanned by the vectors $\xi$ and $\eta$, then there exists $C>5$ such that if $|u(x)+u(x-\xi)+u(x+\eta)|<C^{-n}|u(0)|$ for all $x\in T_{n/2}$, then $|u(x)|>C^{-n}|u(0)|$ for essentially all points of $T_n$. 

\noindent \textit{Step 2.} In the next stage, the result from the first step is used to move to three-dimensional cubes $Q_n$. The following property is also proven and used: for any bounded $q$, if $u$ satisfies $\Delta_{\operatorname{d}}u=qu$ in $Q_n$, then from any point in $Q_n$ there exists a chain of points where $|u|$ decays at most exponentially. This is called ``the cone property''. 

\noindent \textit{Step 3.} 
With the results from the first two steps, it is proven that if  $\Delta_{\operatorname{d}}u=qu$ in $Q_n$, then there exists a sub-plane in $Q_n$ on which $u$ cannot have worse than super-exponential decay: the subset of points such that $|u(x)|\ge \exp[-C_2n^3]|u(0)|$ has cardinality at least $C_3\frac{n^2}{\log n}$. Finally, Li and Zhang find many disjoint translations of $Q_{n^{1/3}}$ inside $Q_n$ and use the unique continuation principle in the previous step to conclude the proof of Theorem \ref{thm:LZ22simp}. As a consequence, Theorem \ref{thm:LZ22simp} can be used to prove a Wegner estimate that is suitable for adapting the strategy of Bourgain and Kenig.

Interestingly, deterministic unique continuation suffices in dimension $n=3$ while in the dimension $n=2$ case considered  by Ding and Smart, a random version is established. The methods in \cite{DS20} have been recently extended in \cite{Hu24}, where a Wegner estimate is proven by using Bernoulli decompositions (see also \cite{Ain09}), to produce a localization result at the bottom of the spectrum. So far, Anderson localization in the discrete setting for dimensions higher than three is an open question. The main difficulty in
proving Anderson localization in $\Z^n$ is the lack of \textit{good} theorems relate to discrete quantitative unique continuation.

\subsection{Landis' conjecture for the Schr\"odinger equation with a discrete Laplacian}

Estimates on the decay of stationary solutions of discrete Schr\"odinger operators
$$
\Delta_{\dis} u_j + q_j u_j=0 \quad   \text{in } \mathbb{Z}^n,
$$
where $q:\Z^n\to \R$ is a bounded potential and $u_j:=u(j)$, and sharp uniqueness results for this equation, were obtained in \cite{LM18} by Y. Lyubarskii and Malinnikova: if $u(x)$ satisfies the following decay estimate
$$
\liminf\limits_{N\rightarrow\infty}\frac{\log(\max_{|j|_{\infty}\in \{N,N+1\}}|u(j)|)}{N}<-\|q\|_{\infty}-4n+1
$$
where $|j|_{\infty}=\max\{j_1,\ldots,j_d\}$, for $j\in \Z^n$, then $u\equiv 0$.

Recently, we proved a result in the spirit of Landis' conjecture for stationary discrete Schr\"odinger equations \cite{FBRS24}, by showing suitable quantitative lower estimates for the $L^2$-norm of the solution within a spatial lattice $(h\Z)^n$, with $h>0$. Set $\Delta_{h,\dis} f_j:=\frac{1}{h^2}\sum_{k=1}^d \big(f_{j+e_k}-2f_j+f_{j-e_k}\big)$, $j\in \Z^n$,
where we denote $f_j:=f(hj)$ and $e_k$ is the unit vector in the $k$-th direction.

\begin{thm}{\cite[Theorem B]{FBRS24}}\label{thm:Schr}
Let $h>0$ and $u\in \ell^2((h\mathbb{Z})^n)$ be a solution to
$$
\Delta_{\dis} u_j + q_j u_j=0
$$
where $\|q\|_{\infty}$, $\|u\|_2$ are finite and independent of $h$.
\begin{enumerate}
\item There exists $\mu_0=\mu_0(n)$ such that if
$$
h^n\sum_{j\in \Z^n} e^{\mu_0 |jh|^{4/3}}  |u_j|^2 <\infty,
$$
uniformly with respect to $h$, then there exists $h_0>0$ such that if $h\in(0,h_0)$, $u\equiv0$.
\item There exists $\mu_0=\mu_0(n)$ such that if, for some $\beta>3$
$$
h^n\sum_{j\in \Z^n}  e^{\mu_0 |jh|^{1+\frac{1}{\beta}}\log |j|h} |u_j|^2 < \infty,
$$
uniformly with respect to $h$, then there exists $h_0>0$ such that $h\in (0,h_0)$ implies $u\equiv0$.  

\item There exists $\mu_0=\mu_0(n)$ such that if
$$
h^n\sum_{j\in \Z^n}  e^{\mu_0 |j|\log |j|h} |u_j|^2 < \infty, \quad \text{ then } u\equiv0.
$$
\end{enumerate}
\end{thm}
 The estimates in Theorem \ref{thm:Schr} manifest an interpolation phenomenon between continuum and discrete scales, showing that close-to-continuum and purely discrete regimes are different in nature. Observe that, when the mesh shrinks to $0$ (which is case (1) in Theorem \ref{thm:Schr}), we get the same $4/3$-exponent as in \cite{Me91,BK05} (i.e., Theorems \ref{thm:Me91} and \ref{thm:BKquantitative}, respectively). Part (3), for which $h$ is considered a fixed parameter, concerns the purely discrete setting, and the result can be compared with the uniqueness result in \cite{LM18}. Our proof relies on a Carleman inequality, so it is conjectured that the result may be sharp if the potential is allowed to be complex. On the other hand, in view of the results in \cite[Subsection 4.4]{LM18}, and the examples in \cite[Subsection 7.3]{FBRS24} (see also \cite[Corollary 7.3]{FBRS24}), it is reasonable to expect that $(2)$ and $(3)$ in Theorem \ref{thm:Schr} are not sharp.
At the moment, it remains uncertain whether $(1)$ in Theorem~\ref{thm:Schr}, concerning the close-to-continuum regime, achieves optimality.  

Even more recently, U. Das, M. Keller, and Y. Pinchover \cite{DKP24} studied Landis' conjecture for positive Schr\"odinger operators $H$ as in \eqref{eq:H} defined on more general graphs. 
More precisely, they gave a decay criterion that ensures when $H$-harmonic functions for a positive Schr\"odinger operator with potentials
bounded from above by $1$ are trivial \cite[Theorem 3.1]{DKP24}. Special cases of $\Z^n$ and regular trees for which they get explicit decay criterion as well as the fractional analogues are investigated.
 Their
approach relies on a discrete version of Liouville comparison principle. 
More details about this approach in the continuum setting will be given in Subsection \ref{sub:LandisLiou}.

\section{Other recent contributions}

There are only a few results that address Landis' conjecture in its original form in higher dimensions \cite{ABG19, R21, SS21}. Landis-type theorems have been studied for linear second-order elliptic operators with real coefficients of the following divergence form
\begin{equation}
\label{eq:elliptico}
\mathcal{L}(u):=-\operatorname{div}\big[\big(A(x)\nabla u+ub(x)\big)\big]+W(x)\cdot \nabla u +q(x)u,\quad x\in \Omega,
\end{equation}
where $A(x):=(a_{ij}(x))_{n\times n}$ is a symmetric, locally uniformly elliptic matrix, $a_{ij}\in L_{\operatorname{loc}}^{\infty}(\Omega)$, $W_i,b_i\in L_{\operatorname{loc}}^{p}(\Omega)$, $p>n$, $q\in L_{\operatorname{loc}}^{r}(\Omega)$, $r>n/2$, and $\Omega\subset \R^n$ is a domain.

For general operators of this type, it was known since Pli\'s \cite[Theorem 2]{Plis63} (see also \cite{Mi74}) that Landis' conjecture has a negative answer:
Pli\'s exhibited an operator in $\R^3$ with a H\"older-continuous matrix field $A$ and smooth terms $W,q$ which admits a nontrivial solution vanishing identically outside a ball. This is a surprising counterexample to the property of unique continuation at infinity. In view of this, the only hope of deriving the UCI is to impose some additional hypotheses on the operator.

\subsection{Landis' conjecture in dimensions one and higher}

L. Rossi \cite{R21} considered a general (real) elliptic operator
\begin{equation}
\label{eq:general}
\widetilde{\mathcal{L}}u:=\operatorname{Tr}(A(x)D^2u)+W(x)\cdot Du+qu,
\end{equation}
in an exterior domain $\Omega\subset \R^n$, where the matrix field $A$ is bounded, continuous and uniformly elliptic. Rossi proved Landis' conjecture in dimension $n=1$, reaching the threshold value $k=\sqrt{|q|}$. Rossi used ODE techniques and a maximum principle approach to treat this one-dimensional case. The latter positive result on Landis' conjecture is then extended for radially symmetric operators in arbitrary dimension, by applying the one-dimensional result to the spherical harmonic decomposition of the solution.  Rossi also extends the result to sign-changing solutions under the assumption that the \textit{generalized principle eigenvalue} $\lambda_0$ is nonnegative, see also \cite{LB21}\footnote{Here, the generalized eigenvalue of an operator $-\mathcal{L}$ is defined as 
$$
\lambda_0:=\sup\{\lambda\in \R: (\mathcal{L}+\lambda)u\le 0 \quad \text{ in } \Omega\}.
$$}.

In \cite{ABG19}, A. Arapostathis, A. Biswas, and D. Ganguly attacked the problem using probabilistic tools. They derive the UCI again under the additional assumption that $u\ge0$ or, if $\Omega=\R^n$, that $\lambda_0\ge 0$, where $\lambda_0$ is the generalized principal eigenvalue.
On the other hand, Landis' conjecture for such general second-order uniformly
elliptic operators, but without the assumption of bounded coefficients, was addressed in \cite{SS21} by means of the comparison principle.
The assumptions in \cite{ABG19, SS21} ensure that the generalized principal eigenvalue of the associated operator $\widetilde{\mathcal{L}}$ is nonnegative.
The threshold for the decay rates $k$ obtained in \cite{ABG19, SS21} depends on the coefficients of $\widetilde{\mathcal{L}}$ and it is not optimal in general. Landis-type results in manifold settings have been analyzed in \cite{PPV24}. 

Finally, there is an important recent result by N. D. Filonov and S. T. Krymskii \cite{FK23} where they consider the equation \eqref{eq:H} in the cylinder $\R\times (0,2\pi)^n$ with periodic boundary conditions. The potential $q$ is assumed to be bounded, and both $u$ and $q$ are assumed to be real-valued. It is shown in \cite{FK23} that the fastest rate of decay at infinity of non-trivial $u$ is $O(e^{-c|w|})$ for $n=1,2$, and $O(e^{-c|w|^{4/3}})$ for $n\ge 3$, where $w$ is the axial variable. This result suggests that the weak version of Landis' conjecture might not be true in dimensions higher than $3$, and the right decay should match with the complex-valued case. 

\subsection{Landis' conjecture via Liouville comparison principle and criticality theory}
\label{sub:LandisLiou}

Very recently, in \cite{DP24}, Das and Pinchover gave new partial affirmative answers to Landis' conjecture in all dimensions for two different types of linear elliptic operators in a domain $\Omega\subset \R^n$. The analogue of Landis' conjecture for quasilinear problems was also addressed. The novelty in \cite{DP24} is an approach relying on the application of Liouville comparison principles for \textit{nonnegative} Schr\"odinger operators, and criticality theory for general \textit{nonnegative} second order elliptic operators in divergence form as in \eqref{eq:elliptico}. As an illustration of their main result is the case of a nonnegative version of the Schr\"odinger operator \eqref{eq:H}. 

\begin{theorem}{\cite[Theorem 1.1]{DP24}}
    \label{thm:daspinchover}
    Let $H$ be a nonnegative Schr\"odinger operator in $\R^n$ as in \eqref{eq:H}, where $n\ge1$, $q\le 1$, and $q\in L^p_{\operatorname{loc}}(\R^n)$, $p>n/2$ if $n\ge 2$, and $p=1$ if $n=1$. If $u\in W_{\operatorname{loc}}^{1,2}$ is a solution to \eqref{eq:H} in $\R^n$ satisfying 
    $$
    |u(x)|=\begin{cases} O(1)\quad &n=1,\\
    O(|x|^{(2-n)/2}) \quad &n\ge 2,
    \end{cases} \quad\text{ as } |x|\to \infty\quad \text{ and } \quad \liminf_{|x|\to\infty} \frac{|u(x)||x|^{(n-1)/2}}{e^{-|x|}}=0,  \quad  \text{ then } u\equiv 0.
    $$
\end{theorem}

Theorem \ref{thm:daspinchover} proves Landis' conjecture in any dimension for nonnegative $H$ with a potential $q$ which can be unbounded from below. The theorem assumes a slower decay rate of $u$ than the one Landis originally proposed in his conjecture. The result does not need the exponential decay of $u$ at infinity in $\R^n$ and it is enough to have a sequence $(x_j)$ with $|x_j|\to \infty$ such that $(u(x_j))$ decays faster than $e^{-|x_j|}/|x_j|^{\frac{n-1}{2}}$. The result is sharp, there is an example to prove it \cite[Example 1.2]{DP24}, which we will describe a few lines below. 

Let us first introduce some definitions. Consider an elliptic operator $\mathcal{L}$ of the form \eqref{eq:elliptico}. A function $u$ is said to be a \textit{positive solution to the equation} $\mathcal{L}u=0$ \textit{of minimal growth at} $\bar{\infty}$ if for some $K\Subset \Omega$, $u\in W_{\operatorname{loc}}^{1,2}(\Omega\setminus K)$, $u$ is a positive solution to the equation $\mathcal{L}u=0$ in $\Omega\setminus K$, and for any positive supersolution $v$ of $\mathcal{L}v=0$ in a subdomain $\Omega\setminus K_1$ with $K\Subset K_1\Subset \Omega$ such that $K_1$ is a Lipschitz bounded domain in $\Omega$, the inequality $u\le v$ on $\partial K_1$ implies $u\le v$ in $\Omega\setminus K_1$. 

Fix $y\in \Omega$. A function $G_{\mathcal{L}}(\cdot, y)\in W_{\operatorname{loc}}^{1,2}(\Omega\setminus \{y\})\cap L_{\operatorname{loc}}^1(\Omega)$ is called a \textit{minimal positive Green function} of $\mathcal{L}$ in $\Omega$ with singularity at $y$, if $G_{\mathcal{L}}(\cdot, y)$ is a positive a solution to the equation $\mathcal{L}u=0$ in $\Omega\setminus\{y\}$ which has a minimal growth at $\bar{\infty}$ in $\Omega$, and satisfies the equation $\mathcal{L}(G_{\mathcal{L}})(\cdot,y)=\delta_y$ in the distributional sense, where $\delta_y$ is the Dirac measure with a charge at $y$. 
A positive solution to $\mathcal{L}u=0$ in $\Omega$ which has a minimal growth at $\bar{\infty}$ is called an \textit{Agmon ground state} of $\mathcal{L}$ in $\Omega$, denoted by $\Psi$. 

The above-mentioned example is as follows. Let $G_{H_1}$ be a minimal positive Green function of the operator $H_1:=-\Delta+1$ in $\R^n$ having a singularity at the origin. It is known (\cite[p.24]{S92} and \cite[Appendix]{DP24}) that $G_{H_1}\simeq |x|^{(1-n)/2}e^{-|x|}$ for $|x|\ge 1$. Let $0< W\in C_0^{\infty}(\R^n)$ and consider the generalized principal eigenvalue of $H_1$.
Then $(H_1-\lambda_0W)u=0$ in $\R^n$ admits a positive solution which is an Agmon ground state $\Psi$. It follows that $\Psi(x)\simeq G_{H_1}(x)\simeq |x|^{(1-n)/2}e^{-|x|}$ for $|x|\ge1$.

It can be realized that the asymptotics at infinity of the functions $|x|^{(1-n)/2}e^{-|x|}$ and $|x|^{(2-n)/2}$ in Theorem \ref{thm:daspinchover} are the same as positive solutions of minimal growth at infinity of the equations $(-\Delta +1)u=0$ and $(-\Delta-W_{\operatorname{loc}})u=0$, respectively, where $W_{\operatorname{opt}}$ is an optimal Hardy-weight of $-\Delta$ in $\R^n$. Precisely, $W_{\operatorname{opt}}\simeq (n-2)^2|x|^{-2}/4$ for $n\ge 3$ (subcritical case) and $W_{\operatorname{opt}}=0$ for $n=1,2$ (critical case). This observation above leads Das--Pinchover to extend Theorem \ref{thm:daspinchover} to nonnegative Schr\"odinger-type operators (i.e., with variable coefficients) in a general domain $\Omega$ in $\R^n$ and to prove a general  \textit{Landis-type theorem} for symmetric operators, see \cite[Theorem 1.3]{DP24}.

Das and Pinchover extend also their result to nonsymmetric linear second-order elliptic operators with real coefficients of divergence form as in \eqref{eq:elliptico}. As explained above, in general, Landis' conjecture may not hold for such an operator \cite{Plis63}, but under suitable restrictions on the coefficients and the crucial assumption $\widetilde{\mathcal{L}}\ge 0$, it is possible to establish Landis-type theorems, see \cite[Theorems 1.4 and 1.5]{DP24}, in the spirit of refined maximum principles \cite[Section 4]{P99}.

\subsection{Landis conjecture and periodic operators}

Another question related to Landis conjecture and mainstreamed by P. Kuchment \cite{Ku12} concerns periodic operators. If we consider an elliptic differential operator of second order $\mathcal{L}u=\operatorname{div}\big(A\nabla u\big) +qu$, whose coefficient are, real, valued, $C^1$-smooth (or even $C^{\infty}$-smooth) in $\R^n$ and periodic with respect to $\Z^n$ translations, we pose the following question: is it true that a solution $\mathcal{L}u=0$ such that
$$
|u(x)|\le e^{-|x|^{\gamma}},
$$
where $\gamma>1$, then $u\equiv0?$.

Related to this question, at the time of writing, still another important contribution by Krymskii, Logunov and F. Pagano has appeared \cite{KLP25}. They construct a real-valued solution to the eigenvalue problem $-\operatorname{div}\big(A\nabla u\big)=\lambda u$, $\lambda>0$ in the cylinder $\mathbb{T}^2\times \R$ with a real, uniformly elliptic, and uniformly $C^1$ matrix $A$ such that $|u(x,y,t)|\le Ce^{-ce^{c|t|}}$ for some $c,C>0$. Their result shows that there is an obstacle to generalize, at least directly, the approach for the continuous Anderson--Bernoulli model by Bourgain and Kenig to operators in divergence form $\mathcal{L}u=\operatorname{div}(A\nabla u)+qu$ because one of the ingredients of the approach fails: quantitative unique continuation properties for operators in divergence form (with variable coefficient of higher order) are substantially weaker than of the operators in the form $\mathcal{L}u=\Delta u +qu$.

Krymskii et al. emphasize  that the current methods
of quantitative unique continuation, such as Carleman inequalities and monotonicity formulas, do not take into account the randomness of $q$ (in the Anderson--Bernoulli model) or the periodicity of $A$ and $q$ in the periodic setting. It would be interesting to develop a method that considers randomness or periodicity of the
coefficients to prove quantitative results ensuring slow decay on large scales.

\subsection{Landis-type results for parabolic and dispersive equations}

In a broader sense, we can consider \textit{Landis-type results} when we are interested in the maximum vanishing rate of solutions to equations with potentials, namely:
\begin{itemize}
\item In the case of elliptic equations, we are concerned about the maximal rate at which nontrivial solutions vanish at infinity. This is related to the original Kondrat'ev--Landis' conjecture.
\item Concerning dispersive or parabolic equations, we are concerned about the maximal spatial decay rate  of nontrivial solutions when time varies within a bounded interval. These are usually called in the literature \textit{dynamical uncertainty principles}.
\end{itemize}

Let us consider the time-dependent Schr\"odinger equation
\begin{equation}
\label{fS}
 \partial_tu(t,x)=i(\Delta u+q(t,x)u).
\end{equation}
In a series of works, L. Escauriaza, Kenig, G. Ponce, and Vega \cite{EKPV-CPDE, EKPV-JEMS, EKPV-DUKE, EKPV-BAMS, EKPV-CMP} proved that if $q$ satisfies one of the following conditions
\begin{itemize}
\item[(i)] $\lim_{R\to\infty}\int_0^T\sup_{|x|>R}|q(t,x)|\,dt=0$
\item[(ii)] $q(t,x)=q_1(x)+q_2(t,x)$, where $q_1$ is real-valued and bounded and for some positive $\alpha$ and $\beta$,
$$ \sup_{[0,T]}\|e^{\alpha\beta T^2|x|^2/(\sqrt{\alpha}t+\sqrt{\beta}(T-t))^2}q_2(t)\|_{L^{\infty}(\R^n)}<+\infty,
$$
\end{itemize}
and we assume that $u$ is a solution to \eqref{fS} which fulfills the decay conditions $ |u(0,x)|\le Ce^{-\alpha|x|^2}$,
$ |u(T,x)|\le Ce^{-\beta|x|^2}
$
with $\alpha\beta>1/(16T^2)$, then $u\equiv0$.  Similarly, for the heat equation with potential, the following result can be deduced: let
$q(t,x)\in L^{\infty}(\R\times \R^n)$ and $u$ be a solution to $\partial_tu=\Delta u+qu$; if
$ |u(T,x)|\le Ce^{-\delta|x|^2}
$
and $\delta>1/(4T)$, then $u\equiv0$. The conditions on $\alpha, \beta$ and $\delta$ are sharp. These results, which exploit a connection with Hardy's uncertainty principle, can be understood as dispersive and parabolic analogues of Landis' conjecture, see also \cite{BFGRV13, CF17}.

Landis-type uniqueness theorems for the time-dependent Schr\"odinger equation with a discrete Laplacian were studied in \cite{FB19, FB20, BFV17, JLMP18}.  The equation under study is
\begin{equation}\label{ph1}
\partial_t u_j=i(\Delta_{\dis} u_j + q_j u_j) \quad   \text{in } \mathbb{Z}^n   \times \R_+,
\end{equation}
where $q:\Z^n\times \R_+\to \R$ is a bounded potential and we use the notation $u_j=u_j(t):=u(j,t)$. It was proved in \cite{JLMP18} (for $n=1$) and in \cite{BFV17} (for arbitrary dimensions) that, if $u$ is a solution to \eqref{ph1} and there exists a constant $\gamma$ such that
$$
|u(j,0)|+|u(j,1)|\le C \exp(-\gamma |j|\log |j|), \quad j\in \Z^n\setminus\{0\},
$$
then $u\equiv0$. From the $q\equiv0$ free case these results are not expected to be sharp. Recently in \cite{FBRS24}, we also proved Landis-type results for the semidiscrete heat in a mesh, studying the interpolation phenomenon between close-to-continuum and discrete regimes.

\subsection{Landis' conjecture for the Dirac equation}

In \cite{Ca22}, B. Cassano determined the largest rate of exponential decay at infinity for non-trivial solutions to the Dirac equation
$$
\mathcal{D}_nu+\mathbb{V}u=0\quad \text{ in } \R^n,
$$
where $\mathcal{D}_n$ is the massless Dirac operator in dimension $n\ge 2$ and $\mathbb{V}$ is a (possibly non-Hermitian) matrix-valued perturbation such that $|\mathbb{V}|\sim |x|^{-\varepsilon}$ at infinity, for $-\infty<\varepsilon< 1$. The result in \cite{Ca22} is sharp for dimensions $n=2,3$, and Cassano provided explicit examples of solutions that exhibit the prescribed decay, in the presence of a potential with the corresponding behaviour at infinity.  In the examples in dimension 2 and 3 in \cite{Ca22}, the constructed potentials are non-symmetric (which corresponds with ``complex-valued'' potentials).
Previous results are due to N. Boussa\"id and A. Comech \cite{BC16}, who studied the point spectrum of the nonlinear massive Dirac equation in any spatial dimension, linearized around one of the solitary wave solutions, and considered the presence of a bounded potential decaying at infinity in a weak sense, showing linear exponential decay for eigenfunctions.

\subsection{Non local versions of Landis' conjecture}

Landis-type conjecture for the fractional Schr\"odinger equation
\begin{equation}
    \label{eq:Landisfrac}
    \big((-\Delta)^s+q\big)u=0\quad \text{ in } \R^n,
\end{equation}
with $s\in (0,1)$ and $|q(x)|\le 1$ was treated by A. R\"uland and Wang in \cite{RW19}. They discuss both qualitative and quantitative estimates when the potentials are either differentiable or simply bounded. 

Concerning qualitative estimates, R\"uland and Wang show that if $s\in (0,1)$ and $u\in H^s(\R^n)$ is a solution to \eqref{eq:Landisfrac}, where $q\in C^1(\R^n)$ and $|x\cdot \nabla q(x)|\le 1$, and moreover $u$ satisfies the decay condition
\begin{equation}
\label{eq:decayfrac}
    \int_{\R^n}e^{|x|^{\alpha}}|u|^2\,dx\le C<\infty
\end{equation}
for some $\alpha>1$, then $u\equiv 0$. Without the additional regularity on $q$, under the assumption $s\in (\frac14,1)$ and that $u$ satisfies \eqref{eq:decayfrac} for some $\alpha>\frac{4s}{4s-1}$, the same conclusion holds. On the other hand, R\"uland and Wang give a quantitative version of the latter result, that is, if $s\in (\frac14,1)$ and $u\in H^s(\R^n)$ is a solution to \eqref{eq:Landisfrac} such that $\|u\|_{L^2(B_1)}=1$ and $\|u\|_{L^{\infty}(\R^n)}\le C_0$ for some $C_0>0$, then there exists $C=C(n,s,C_0)>0$ such that 
$$
  \inf_{|x_0|=R}\sup_{|x-x_0|<1}|u(x)|\ge C \exp[-CR^{\frac{4s}{4s-1}}\log R]\quad \text{ for } R\gg 1.
$$
The approach used is the so-called Caffarelli--Silvestre $s$-harmonic extension \cite{CS07} and suitable Carleman-type estimates. For the quantitative estimates, they use the strategy of Bourgain and Kenig \cite{BK05} explained in Section \ref{BK}. Other Landis-type results related to non-local operators, in the spirit of Caffarelli--Silvestre extension approach,  have been obtained in \cite{K22, KW23}. Moreover, a fractional analogue of Landis' conjecture on $\Z^n$ was recently investigated in \cite{DKP24}.

\end{document}